\documentclass{amsart}
\usepackage[utf8]{inputenc}
\usepackage[a4paper,margin=2.5cm]{geometry}

\usepackage{amssymb}
\usepackage{amsthm}
\usepackage{mathtools}
\usepackage{mleftright}
\usepackage[shortlabels]{enumitem}

\usepackage{booktabs}
\newcommand\headercell[1]{\smash[b]{\begin{tabular}[t]{@{}c@{}} #1 \end{tabular}}}
\usepackage{array}
\newcolumntype{M}[1]{>{\centering\arraybackslash}m{#1}}
\usepackage{longtable}
\usepackage{multirow}
\usepackage{lscape}

\usepackage{pgfplots}
\pgfplotsset{width=7cm,compat=1.9}
\usepackage{subfig}
\usepackage{tikz}
\usetikzlibrary{shapes.misc}
\graphicspath{{./Plots/}}
\tikzset{cross/.style={cross out, draw=black, minimum size=2*(#1-\pgflinewidth), inner sep=0pt, outer sep=0pt}, cross/.default={1pt}}
\usetikzlibrary{shapes.geometric}

\usepackage[ruled,vlined]{algorithm2e}

\usepackage[colorinlistoftodos, textwidth=2.5cm,textsize=small]{todonotes}

\usepackage{hyperref}
\hypersetup{%
    pdffitwindow=false,     
    pdfstartview={FitH},    
    pdftitle={Half-integral polytopes with a fixed number of lattice points},    
    pdfauthor={The authors go here},     
    pdfkeywords={Ehrhart Theory} {Lattice Polytopes} {Lattice Width}, 
    pdfnewwindow=true,      
    colorlinks=true,       
    linkcolor=blue,          
    citecolor=red,        
    filecolor=magenta,      
    urlcolor=cyan           
}

\theoremstyle{plain}
\newtheorem{theorem}{Theorem}[section]
\newtheorem{proposition}[theorem]{Proposition}
\newtheorem{lemma}[theorem]{Lemma}
\newtheorem{corollary}[theorem]{Corollary}

\theoremstyle{definition}
\newtheorem{definition}[theorem]{Definition}
\newtheorem{remark}[theorem]{Remark}

\newcommand{\NN}{{\mathbb{N}}}
\newcommand{\QQ}{{\mathbb{Q}}}
\newcommand{\RR}{{\mathbb{R}}}
\newcommand{\ZZ}{{\mathbb{Z}}}

\newcommand{\mycomment}[1]{}

\DeclareMathOperator{\cone}{cone}
\DeclareMathOperator{\pen}{pen}
\DeclareMathOperator{\conv}{conv}

\DeclareMathOperator{\interior}{int}
\DeclareMathOperator{\rec}{rec}
\DeclareMathOperator{\width}{width}
\DeclareMathOperator{\Vol}{Vol}

\DeclareMathOperator{\ehr}{ehr}

\title{Classification and Ehrhart Theory of Denominator 2 Polygons}

\author[G.\ Hamm]{Girtrude~Hamm}
\address{The Department of Mathematics\\University of Western Ontario\\
London\\ Ontario\\ N6A 5B7\\Canada}
\email{ghamm@uwo.ca}
\author[J.\ Hofscheier]{Johannes~Hofscheier}
\address{School of Mathematical Sciences\\University of Nottingham\\Nottingham\\NG7 2RD\\UK}
\email{johannes.hofscheier@nottingham.ac.uk}
\author[A.\ Kasprzyk]{Alexander~Kasprzyk}
\address{School of Mathematical Sciences\\University of Nottingham\\Nottingham\\NG7 2RD\\UK}
\email{a.m.kasprzyk@nottingham.ac.uk}


\begin{document}

\begin{abstract}
    We present an algorithm for growing the denominator \(r\) polygons containing a fixed number of lattice points and enumerate such polygons containing few lattice points for small \(r\).
    We describe the Ehrhart quasi-polynomial of a rational polygon in terms of boundary and interior point counts.
    Using this, we bound the coefficients of Ehrhart quasi-polynomials of denominator 2 polygons.
    In particular, we completely classify such polynomials in the case of zero interior points.
\end{abstract}

\maketitle

\section{Introduction}

A \emph{lattice point} is a point in a lattice \(\ZZ^d\).
A \emph{rational polytope} \(P \subseteq \RR^d\) is the convex hull of finitely many rational points \(v_1, \dots, v_k\) in \(\QQ^d\).
It is called a \emph{lattice polytope} if its vertices are all lattice points, and we call a polytope a \emph{polygon} if it is two-dimensional.
We consider rational polytopes as being defined up to affine unimodular equivalence.
Polytopes \(P\) and \(P'\) are said to be \emph{affine equivalent} if \(P\) can be mapped to \(P'\) by a change of basis of \(\ZZ^d\) followed by an integral translation.
The \emph{denominator} of a rational polytope \(P\) is an integer \(r \geq 1\) such that the dilation \(rP\) is a lattice polytope.
Note that some authors require this to be the minimal such integer, but we do not include that requirement.
The \emph{size} of a polytope is the number of lattice points it contains.
In this paper we algorithmically classify rational polygons with small size and denominator and use the classification to study the Ehrhart theory of denominator 2 polygons.

Rational polygons appear in a variety of settings, making them interesting objects to study in their own right.
For example, spherical Fano varieties of rank 2 correspond to certain rational polygons \cite{GG_JH_sph_poly}. 
In \cite{k-empty-polygons} rational polygons whose only lattice points are vertices, are classified to understand canonical three-fold singularities with a complexity one torus action.
Denominator 2 polygons are the intersection of a 3-dimensional width 2 polytope with a hyperplane, which is used in \cite{half-int-intersections}, towards understanding the three-dimensional maximal polyhedra with no interior lattice points.
Also of note is the paper of Bohnert and Springer \cite{Bohnert-Justus-Class}, classifying rational polygons by number of interior points rather than size, which was completed concurrently with this project.
Their results agree with and extend Table~\ref{tab:number_found}, though their methods are different to ours.

Fix integers \(r \geq 2\) and \(k \geq 0\), then there are infinitely many rational polygons with denominator \(r\) and size \(k\), up to affine equivalence.
For example, when \(a\) and \(k\) are positive integers the polygon \(\conv((0,0),(k-1,0),(0,\frac1r),(\frac{a}{r},\frac1r))\) has denominator \(r\) and size \(k\) as in Figure~\ref{fig:inf_growable}.
However, all but finitely many denominator \(r\) size \(k\) polygons are part of an infinite family in the following sense.
\begin{definition}\label{def:inf_grow}
A polygon \(P\) with denominator \(r\) and size \(k\) is \emph{infinitely growable} if there exists an infinite sequence of polygons \(P_0\), \(P_1\), \(P_2\), \(\dots\) such that
\begin{enumerate}
    \item \(P_0=P\),
    \item \(P_i\) is a proper subset of \(P_{i+1}\) for all \(i=0,1,2,\dots\) and
    \item \(P_i\) is a polygon with denominator \(r\) and size \(k\) for all \(i=0,1,2,\dots\).
\end{enumerate} 

A polygon \(Q\) with denominator \(r\) and size \(k\) is said to be \emph{maximal} if for given any point \(v \in \frac{1}{r}\ZZ^2\) not in \(Q\) the size of \(\conv(Q \cup v)\) is greater than \(k\).
We say that a polygon \(P\) with denominator \(r\) and size \(k\) is \emph{finitely growable} if it is not infinitely growable.
Notice that a finitely growable polygon is always a subset of some maximal polygon.
\end{definition}

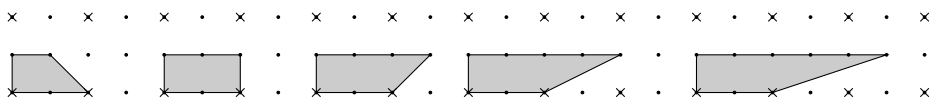
\begin{figure}[ht]
\centering
\begin{tikzpicture}[x=0.5cm,y=0.5cm]
\foreach \i in {0,...,3}{
    \draw[fill=gray!40] (4*\i,0) -- (4*\i+2,0) -- (4*\i+\i+1,1) -- (4*\i,1) -- cycle;
}
\draw[fill=gray!40] (18,0) -- (20,0) -- (23,1) -- (18,1) -- cycle;
\foreach \x in {0,2,...,24}{
	\foreach \y in {0,2}{
		\node[cross=2pt] at (\x,\y) { };
	}
}
\foreach \x in {0,...,24}{
    \foreach \y in {0,1,2}{
        \node[draw,circle,inner sep=0.1pt,fill] at (\x,\y) { };
    }
}
\end{tikzpicture}
\caption{A collection of polygons with denominator 2 and size 2 which can be grown infinitely.
Crosses denote points of \(\ZZ^2\), dots denote points of \(\frac12\ZZ^2\).}
\label{fig:inf_growable}
\end{figure}
We prove the following result in Corollary~\ref{cor:dim-2-grow-inf} and Proposition~\ref{prop:hip_fin}.
\begin{proposition}\label{prop:fin_many_fin_grow}
    Let \(P\) be a denominator \(r\), size \(k\) polygon, then either \(P\) is infinitely growable and is equivalent to a polygon contained in the strip \([0,1] \times \RR\) or \(P\) is finitely growable and is one of finitely many exceptions.
\end{proposition}

Since the finitely growable denominator \(r\) size \(k\) polygons are finite in number they are a potential subject for classification.
We classify them for small \(r\) and \(k\) using a growing algorithm, which involves finding a set of minimal polygons and iteratively adding points to them.
We implement the algorithm in \textsc{Magma~V2.27} and the resulting data can be found at \cite{data}.
The polygons are enumerated in Table~\ref{tab:number_found}.

\begin{table}[ht]
    \centering
    \begin{tabular}{@{} ccccccc @{}}
    \headercell{} & \multicolumn{6}{c@{}}{\(k\)}\\
    \cmidrule(l){2-7}
    \(r\) & 0 & 1 & 2 & 3 & 4 & 5 \\
    \midrule
    1 & 0 & 0 & 0 & 1 & 3 & 6\\
    2 & 1 & 106 & 1333 & 8774 & 40139 & ?\\
    3 & 211 & ? & ? & ? & ? & ?\\
    \end{tabular}
    \caption{The number of finitely growable polygons with denominator \(r\) and size \(k\).}
    \label{tab:number_found}
\end{table}

The Ehrhart quasi-polynomial is an important affine invariant of a rational polytope, counting the number of lattice points in its integral dilations.
For lattice polygons study of the Ehrhart quasi-polynomial reduces to the study of pairs \((b(P),i(P))\), where \(b(P)\) is the number of boundary points of \(P\) and \(i(P)\) is the number of interior points.
Using this fact, the Ehrhart quasi-polynomials of lattice polygons have been completely classified \cite{Scott,Haase_Schicho} (see Theorem~\ref{thm:Scotts}).
In higher dimensions and denominators no such complete classification is known, though some progress has been made towards understanding the denominator 2 polygons case.
In their master's thesis, Herrmann approached a classification of the odd components of these quasi-polynomials \cite{Herrmann}.
Bohnert and Springer have also found bounds on the coefficients of these polynomials, in their paper created concurrently with this one \cite{Bohnert-Justus-Ehr}.
In particular, they give bounds on the volume and an upper bound on the number of boundary points.

In a similar way to the lattice polygon case, study of Ehrhart polynomials of denominator 2 polygons reduces to the study of tuples \((b(P),i(P),b(2P),i(2P))\).
We prove the following result, describing restrictions on this four-dimensional set.
\begin{theorem}\label{thm:hip_finlap_main}
    For all but finitely many denominator 2 polygons \(P\), the number of boundary and interior points of \(P\) and \(2P\) satisfy one of the following conditions:
    \begin{enumerate}[(a)]
        \item \(i(P)=0\), \(i(2P)=0\) and \(b(2P) \geq \max(3,2b(P))\),
        \item \(b(P)=0\), \(i(P)=0\), \(b(2P)=4\) and \(i(2P)>0\),
        \item \(i(P)=0\), \(i(2P),b(P)>0\), \(\max(3,2b(P)) \leq b(2P) \leq 2b(P)+4\) and \(b(2P) \leq 2i(2P)+6\) or
        \item \(i(P) > 0\), \(b(2P) \geq \max\{3,2b(P)\}\), \(i(2P) \geq b(P)+2i(P)-1\) and \(b(2P)+i(2P) \leq 2b(P) + 6i(P) +7\).
    \end{enumerate}
\end{theorem}
When \(i(P)=0\) there are exactly 2 polygons which satisfy none of these conditions, and all tuples \((b_1,0,b_2,i_2) \in \ZZ^4\) which satisfy condition (a), (b) or (c) are realised by some polygon.
Thus, we have a complete classification of the Ehrhart polynomials of denominator 2 polygons with zero interior points.

Condition (d) does not completely classify the remaining tuples in the same way.
There are points which satisfy condition (d) and are realised by no denominator 2 polygon, and there may be finitely many exceptions to the inequality \(b(2P)+i(2P) \leq 2b(P) + 6i(P) +7\), though we have found none.
Our interest in these particular bounds is justified by the existence of certain infinite families of polygons shown in Figure~\ref{fig:inf_families}.
These realise infinite collections of points on the boundaries of (d), contained in pairs of skew lines.
This suggests that an `optimal' polyhedron, containing all but finitely many points of the form \((b(P),i(P),b(2P),i(2P))\), will have unbounded facets defined by these inequalities.

In Section~\ref{sec:infty-grow-hip} we describe infinitely growable polytopes in greater generality and show that a rational polygon is infinitely growable if and only if it is equivalent to a subset of the strip \([0,1] \times \RR\).
In Section~\ref{sec:growing_alg} we present an algorithm for finding all the finitely growable polygons and prove that it will always terminate in finite time.
This depends on finding a list of `minimal' polygons and growing them by successively adding points.
In Section~\ref{sec:minimal_polygons} we classify the minimal, denominator \(r\), size \(k\) polygons in general.
When \(k>0\), these are the lattice polygons of size \(k\) and one rational triangle.
Using these minimal polygons, we run the classification for small values of \(r\) and \(k\).
Finally, in Section~\ref{sec:ehrhart} we use this dataset to investigate the Ehrhart theory of denominator 2 polygons and prove Theorem~\ref{thm:hip_finlap_main}.
This includes a complete classification of the tuples \((b(P),i(P),b(2P),i(2P))\) when \(P\) has no interior points.

\section{Infinitely Growable Polygons}
\label{sec:infty-grow-hip}

In this section we show that a rational polygon is infinitely growable if and only if it is equivalent to a subset of the strip \([0,1] \times \RR\).
We work in greater generality; instead of \(\frac1r\ZZ^2\) and \(\ZZ^2\) we consider a lattice \(L\) and sublattice \(K \subseteq L\) both of rank \(d\).
For a rank \(d\) lattice \(K\) we denote \(K \otimes_\ZZ \RR \cong \RR^d\) by \(K_\RR\).
We think of both \(K\) and \(L\) as sublattices of the Euclidean space \(K_\RR = \RR^d\) where \(K\) is the integral lattice \(\ZZ^d\) and \(L\) is a bigger lattice extended by rational points.

For fixed lattices \(L\) and \(K\) recall that there exists a basis \(\ell_1, \dots, \ell_d\) of \(L\) and unique positive integers \(a_1, \dots, a_d\) with \(a_i \mid a_{i+1}\) for \(i=1, \dots, d-1\) such that \(\kappa_1 \coloneqq a_1\ell_1, \dots, \kappa_d \coloneqq a_d\ell_d\) is a basis of \(K\).
The \(a_i\) are usually referred to as \emph{elementary divisors} (or \emph{invariant factors}) of the sublattice \(K \subseteq L\). Notice that \(a_i\ge 1\) for all \(i=1, \dots, d\) as \(K\) has full rank.
Furthermore, observe that we have a chain of sublattices \(K \subseteq L \subseteq \frac1{a_d}K\) where \(\frac1{a_d}K\) is the sublattice of \(L_\RR\) generated by the basis \(\frac1{a_d}\ell_1, \dots, \frac1{a_d}\ell_d\).
The definition of \(\frac1{a_d}K\) is independent of the choice of basis.

Let us recall some important concepts from \cite{LattProj} with adjusted notation which allows us to consider multiple lattices.

\begin{definition}
Let \(N\) be a lattice and \(P\) a polytope in \(N_\RR\).
We call \(P\) an \emph{\(N\)-polytope} if the vertices of \(P\) are contained in \(N\).
We say that \(P\) is \emph{\(N\)-hollow} if it doesn't contain any points of \(N\) in its interior.
For \(s \in \ZZ_{>0}\), we say a convex body \(P \subseteq N_\RR\) is \emph{\(s\)-\(N\)-hollow} if the interior of \(P\) doesn't contain any points from the dilated lattice \(sN\).
We say that \(P\) is \emph{of \(N\)-size \(k\)} if \(|P \cap N| = k\).
\end{definition}
The definition of infinitely growable and finitely growable extend naturally to this setting by considering \(L\)-polytopes instead of polytopes with denominator \(r\) and \(K\)-size instead of size.

A \emph{lattice projection} is a surjective affine-linear map \(\varphi\colon N_\RR \to V\) onto a vector space \(V\) of dimension \(m\) whose kernel is generated, as a vector space, by elements of \(N\).
Given a lattice projection \(\varphi\) from \(L_\RR=K_\RR\) we get a chain of sublattices 
\[
\varphi(K) \subseteq \varphi(L) \subseteq \varphi\left(\tfrac{1}{a_d}K\right) = \tfrac{1}{a_d}\varphi(K).
\]
We adapt the following result from \cite[Theorem~2.1]{LattProj}.
\begin{proposition}\label{prop:hollow}
    There are only finitely many \(K\)-hollow, \(d\)-dimensional \(L\)-polytopes that do not admit a lattice projection \(\varphi\) to a \((d-1)\)-dimensional \(\varphi(K)\)-hollow polytope.
\end{proposition}
\begin{proof}
    Any \(L\)-polytope is also a \(\frac{1}{a_d}K\)-polytope, so it suffices to show the result in the case where \(L=\frac{1}{a_d}K\).
    By replacing \(K\)-hollow with \(a_d\)-\(\frac{1}{a_d}K\)-hollow and \(\varphi(K)\)-hollow with \(a_d\)-\(\frac{1}{a_d}\varphi(K)\)-hollow this is exactly \cite[Theorem~2.1]{LattProj}.
\end{proof}

\begin{proposition}\label{prop:grow_inf_hollow}
    An infinitely growable \(L\)-polytope \(P \subseteq L_\RR\) is \(K\)-hollow and it admits a lattice projection \(\varphi\) to a \((d-1)\)-dimensional \(\varphi(K)\)-hollow polytope.
\end{proposition}
Towards a proof of Proposition~\ref{prop:grow_inf_hollow}, recall that the Euclidean space \(\RR^d\) can be equipped with the maximum norm which associates to \(x \in \RR^d\) the maximal absolute value of its coordinates, i.e., \(\|x\|_\infty = \max_i |x_i|\).
We will use the following well-known result (see, for instance, \cite[Theorem~1B]{Schmidt}).
\begin{theorem}[Dirichlet's approximation theorem]\label{thm:Dirichlet}
	For every \(w \in \RR^d\) and every \(n > 1\) there exist \(k \in \ZZ\) with \(1 \le k \le n^d\) and \(x \in \ZZ^d\) such that \(\|kw - x\|_\infty < 1/n\).
	In other words, for every ray in \(\RR^d\) there exist lattice points that lie arbitrarily close to it.
\end{theorem}
Also, recall the \emph{recession cone} (or \emph{tail cone}) of a nonempty closed convex set \(C \subseteq \RR^d\) is defined as
\[
\rec(C) \coloneqq \{u \in \RR^d \mid u+C \subseteq C\}.
\]
\begin{proof}[Proof of Proposition \ref{prop:grow_inf_hollow}]
    Assume towards a contradiction that there exists \(v \in \interior(P) \cap K\).
    Since \(P\) can be grown infinitely, there exists a strictly increasing sequence \((P_i)_{i \in \NN}\) of \(L\)-polytopes that have \(K\)-size \(|P_i \cap K| = |P \cap K|\) such that \(P \subseteq P_i\) and \(P_i \subsetneq P_{i+1}\).	
    The closure of the union of the \(P_i\), that is
    \[
    C \coloneqq \overline{\bigcup_iP_i},
    \]
    is an unbounded, closed, convex set.
    Hence, its recession cone is non-trivial, i.e. there exists a non-zero point \(w\) in \(\rec(C)\).
    By Theorem~\ref{thm:Dirichlet}, there exist lattice points in \(K\) that lie arbitrarily close to the ray \(\RR_{\ge0}w\).
    Since \(v \in K\), there also exist lattice points in \(K\) that lie arbitrarily close to the affine ray \(\rho \coloneqq v + \RR_{\ge0}w\).
    Notice that either \(\rho\) contains a second lattice point of \(K\) or \(v\) is the only lattice point of \(K\) that is contained in \(\rho\).
    In both cases we can conclude that there are infinitely many lattice points in \(K\) that are contained in the interior of \(C\).
    However, \(\interior(C) \subseteq \bigcup_i P_i\), and thus the sizes of the \(P_i\) become arbitrarily large.
    This is a contradiction which proves that \(P\) is \(K\)-hollow.
	
    Each of the \(P_i\) is also infinitely growable using the sequence of polytopes \((P_j)_{j\geq i}\) so the \(P_i\) are also hollow.
    By Proposition~\ref{prop:hollow} there is an \(i \in \NN\) such that \(P_i\) admits a lattice projection \(\varphi\) to a \((d-1)\)-dimensional \(\varphi(K)\)-hollow polytope.
    Since \(P\subseteq P_i\) this projection also takes \(P\) to a \((d-1)\)-dimensional \(\varphi(K)\)-hollow polytope.
\end{proof}

From now on we return to the setting where \(K=\ZZ^2\) and \(L=\frac1r\ZZ^2\).

\begin{corollary}\label{cor:dim-2-grow-inf}
    Let \(P\) be a rational polygon with denominator \(r\) and size \(k\).
    Then \(P\) is infinitely growable if and only if it is equivalent to a subset of the strip \([0,1] \times \RR\) of the plane.
\end{corollary}
\begin{proof}
Suppose \(\varphi\) is an affine map such that \(\varphi(P)\) is a subset of \([0,1] \times \RR\).
Then for a sufficiently large integer \(a\), the polygons \(P_i\coloneqq\conv(P,\varphi^{-1}(\frac1r,\frac{a+i}{r}))\) for \(i=0,1,\dots\) form an infinite sequence of denominator \(r\) size \(k\) polygons realising \(P\) as infinitely growable.

Now suppose \(P\) is infinitely growable.
By Proposition~\ref{prop:grow_inf_hollow} there is a lattice projection \(\varphi: \RR^2 \to \RR\) such that \(\varphi(P)\) is \(\varphi(\ZZ^2)\)-hollow.
By a scaling we may assume that \(\varphi(\ZZ^2)=\ZZ\) so in other words \(\varphi(P)\) is a subset of an interval \([a,a+1]\) for some integer \(a\).
The map \(\varphi\) is equal to the map induced by the dual lattice vector \((\varphi(1,0), \varphi(0,1))\).
By a change of basis we may assume that \(\varphi(0,1) = 0\) (and \(\varphi(1,0) \neq 0\)).
Then, the fact that \(\varphi(P) \subseteq [a,a+1]\) can be reinterpreted as: the \(x\)-coordinates of \(P\) are in the range \(\left[\frac{a}{\varphi(1,0)},\frac{a+1}{\varphi(1,0)}\right]\).
By a translation, we may assume that the \(x\)-coordinates of \(P\) are in \([0,1]\).
\end{proof}

\section{Growing Finitely Growable Polygons}
\label{sec:growing_alg}

We classify the finitely growable polygons of small size and denominator using a growing algorithm.
This is done by finding a collection of minimal polygons from which all others can be obtained by successive adding of points.

\begin{definition}
    A polygon \(P\) with denominator \(r\) and size \(k\) is called \emph{minimal} if, for each vertex \(v\) of \(P\), the polygon
    \[
    \conv\left(\left(P \cap \tfrac1r\ZZ^2\right)\setminus\{v\}\right)
    \]
    either has size less than \(k\) or is less than two-dimensional.
\end{definition}

The growing algorithm is based on the following result.

\begin{proposition}\label{prop:growing_alg}
    Let \(P\) be a polygon with denominator \(r\) and size \(k\).
    Then there exists a sequence of polygons \(P_0, P_1, \dots, P_s\) with denominator \(r\) and size \(k\) such that \(P_0\) is minimal, \(P_s=P\) and for all \(i=0,\dots,s-1\):
    \begin{enumerate}
        \item \(P_{i+1} = \conv(P_i, v_i)\), for some point \(v_i\) in \(\frac1r\ZZ^2\),
        \item there is exactly one more lattice point in \(rP_{i+1}\) than in \(rP_i\) (that is \(rv_i\)) and
        \item \(v_i\) is contained in a hyperplane defined by \(u_i \cdot x = h_i+\frac1r\), where \(u_i\) is the outwards pointing normal vector of a facet of \(P_i\) and that facet is a subset of the hyperplane \(u_i \cdot x = h_i\).
    \end{enumerate}
\end{proposition}

We call a hyperplane like the one containing \(v_i\) in point (3) a \emph{hyperplane adjacent to a facet of \(P\)}.
To prove the proposition we need the following preparatory lemma which is a consequence of Pick's Theorem.
By \emph{lattice length} of a lattice line segment, we mean the number of lattice points it contains minus 1.

\begin{lemma}\label{lem:Picks_point_on_line}
    For lattice \(N\cong\ZZ^2\) let \(P\) be a rational polygon in \(N_\QQ\). 
    Suppose \(P\) contains a point of \(N\) in the interior of the half-space 
    \[
    H\coloneqq \{v \in N_\RR : u\cdot v \geq h\}
    \]
    where \(u\) is a primitive element of the dual space to \(N\).
    If \(P\) also contains a lattice line segment of lattice length 1 in the boundary of \(H\) then \(P\) contains a lattice point in the hyperplane \(\{v : u\cdot v = h+1\}\). 
\end{lemma}
\begin{proof}
    By an affine transformation we may assume that \(u=(0,1)\) so that \(H\) is the set of points with non-negative \(y\)-coordinates.
    By a translation we may assume that \(P\) contains the line segment \(\conv((0,0),(1,0))\).
    Let \(v_0=(x_0,y_0)\) be the lattice point in \(P \cap H^\circ\).
    If \(y_0=1\) we are done, otherwise consider the lattice triangle \(T_0 = \conv((0,0),(1,0),v_0)\) contained in \(P\).
    This triangle has volume \(y_0>1\) so by Pick's theorem contains more than three lattice points.
    Let \(v_1=(x_1,y_1)\) be a lattice point in \(T_0\) which is not a vertex.
    If \(y_1=1\) we are done, otherwise replace \(v_0\) with \(v_1\) and repeat this process until we obtain a lattice point in \(P\) with \(y\)-coordinate 1.
    This process terminates since each successive vertex \(v_i\) has strictly smaller \(y\)-coordinate.
\end{proof}

\begin{proof}[Proof of Proposition~\ref{prop:growing_alg}]
If \(P\) is minimal we are done.
If \(P\) is not minimal let \(v_{s-1}\) be a vertex of \(P\) such that \(P_{s-1} \coloneqq \conv(P \cap \frac1r\ZZ^2 \setminus \{v_{s-1}\})\) is also a denominator \(r\) size \(k\) polygon.
It is immediate that \(P = \conv(P_{s-1},v_{s-1})\) and that there is exactly one more lattice point in \(rP\) than in \(rP_{s-1}\).
It remains to show that \(v_{s-1}\) is contained in a hyperplane adjacent to a facet of \(P_{s-1}\).
    
We can represent \(P_{s-1}\) as the intersection of a finite number of half-spaces, each corresponding to a facet of \(P_{s-1}\).
Since \(v_{s-1} \notin P_{s-1}\) one of these half-spaces does not contain \(v_{s-1}\).
Since a facet of \(P_{s-1}\) is contained in the boundary of this half-space \(P\) contains both a point in the half-space and a line segment in its boundary.
Therefore, by Lemma~\ref{lem:Picks_point_on_line} there is a point \(P\) which falls on a hyperplane adjacent to a facet of \(P_{s-1}\).
However, there is only one point of \(\frac1r\ZZ^2\) in \(P\) and not \(P_{s-1}\), that is \(v_{s-1}\).
This shows that \(v_{s-1}\) satisfies the conditions of the proposition.

If \(P_{s-1}\) is minimal we are done, otherwise continue by induction.
The process terminates since \(P\) contains a finite number of points of \(\frac1r\ZZ^2\).
\end{proof}

The algorithm will start with the list of minimal polygons and successively add points which lie on hyperplanes adjacent to each of their facets. 
Based on Proposition~\ref{prop:growing_alg} all finitely growable polygons will certainly occur in this manner.
However, we must also ensure that the algorithm terminates.

There are infinitely many points on each hyperplane so we need to bound the collection of points which we add to polygons at each growing step.
The following definition is useful for this.
\begin{definition}
    Let \(P \subseteq \RR^n\) be a polytope with vertices \(v_1,\dots,v_r\) and \(v \in \QQ\) a point.
    Define the \emph{penumbra} of \(v\) with respect to \(P\) to be
    \[
    \pen(P,v) \coloneqq v - \cone(P-v) = \left\{v - w : w = \sum_{i=1}^r \lambda_i(v_i-v), \lambda_i \in \RR_{\geq0}\right\}.
    \]
    Let \(\pen(P,w_1,\dots,w_r)\) denote the union of the affine cones \(\pen(P,w_1)\), \dots, \(\pen(P,w_r)\).
\end{definition}
See Figure~\ref{fig:penumbra} for an example of a penumbra.
\begin{proposition}
    A point \(x\) is in \(\pen(P,v)\) if and only if \(v\) is in \(\conv(P,x)\).
\end{proposition}
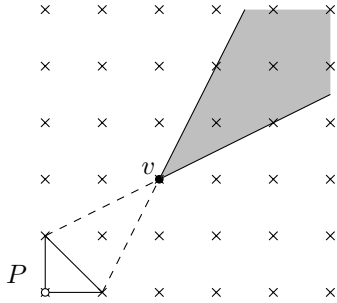
\begin{figure}[ht]
\centering
\begin{tikzpicture}[x=0.75cm,y=0.75cm]
\draw[] (0,0) -- (1,0) -- (0,1) -- cycle;
\node[draw,circle,inner sep=1pt,fill] at (2,2) {};
\fill[fill=gray!50] (5,3.5) -- (2,2) -- (3.5,5) -- (5,5) -- cycle;
\draw (2,2) -- (3.5,5);
\draw (2,2) -- (5,3.5);
\draw[dashed] (1,0) -- (2,2);
\draw[dashed] (0,1) -- (2,2);
\foreach \x in {0,...,5}{
	\foreach \y in {0,...,5}{
		\node[cross=2pt] at (\x,\y) {};
	}
}
\node[draw,circle,inner sep=1pt,fill=white] at (0,0) {};
\node[anchor=south] at (-0.5,0) {\(P\)};
\node[anchor=east] at (2.1,2.2) {\(v\)};
\end{tikzpicture}
\caption{For \(P=\conv((0,0),(1,0),(0,1))\) and \(v=(2,2)\) the shaded region is the penumbra of \(v\) with respect to \(P\), that is \(\pen(P,v) = (2,2) + \cone((2,1),(1,2))\).}
\label{fig:penumbra}
\end{figure}
\begin{proof}
    Let \(x \in \pen(P,v)\). By definition there exist non-negative rational numbers \(\lambda_1,\dots,\lambda_r\) such that
    \[
    x = v - \sum_{i=1}^r\lambda_i(v_i-v)
    \]
    which we can rearrange to write \(v\) as
    \[
    v = \frac{1}{1+\sum_{i=1}^r\lambda_i}\left( x + \sum_{i=1}^r \lambda_iv_i\right).
    \]
    This means that \(v\) meets the conditions to be an element of \(\conv(P,x)\). 
    Now let \(v \in \conv(P,x)\), then there exist non-negative rational numbers \(\mu_0,\dots,\mu_r\) such that \(\mu_0+\dots+\mu_r=1\) and 
    \[
    v = \mu_0x + \sum_{i=1}^r\mu_iv_i
    \]
    which we can rearrange to write \(x\) as
    \[
    x = \frac{1}{\mu_0}\left(v-\sum_{i=1}^r\mu_iv_i\right) = v - \sum_{i=1}^r \frac{\mu_i}{\mu_0}(v_i-v)
    \]
    if \(\mu_0 \neq 0\) so \(x \in \pen(P,v)\).
    If \(\mu_0=0\) then \(v \in P\) so \(\pen(P,v)\) is the whole space \(\RR^n\) and contains \(x\).
\end{proof}

Suppose we wish to add a point \(v\) to \(P\) on the hyperplane adjacent to a facet \(F\) of \(P\) in a way which satisfies the conditions of Proposition~\ref{prop:growing_alg}.
Since \(\conv(P,v)\) has exactly one additional point, it must contain no new points in the hyperplane \(H\) containing \(F\).
In particular, let \(x_1\) and \(x_2\) be two distinct points of \(\frac1r\ZZ^2\) in \(H\) and not \(F\), which are each as close to a vertex of \(F\) as possible.
Then \(v\) cannot be in \(\pen(P,x_1,x_2)\).
This restricts us to a finite collection of points \(v\) in the hyperplane adjacent to \(F\).
For an example of the points we can add see Figure~\ref{fig:growing_eg}.
We exclude lattice points to avoid increasing the size.

\begin{figure}[ht]
\centering
\begin{tikzpicture}[x=0.6cm,y=0.6cm]
    \draw (0,0) -- (1,0) -- (0,1) -- cycle;
    \fill[gray!50] (-1,0) -- (-6,0) -- (-6,-2) -- (-3,-2) -- cycle;
    \fill[gray!50] (2,0) -- (6,0) -- (6,-2) -- cycle;
    \draw[dashed] (0,0) -- (-1,0);
    \draw[dashed] (1,0) -- (2,0);
    \draw[dashed] (0,1) -- (-1,0);
    \draw[dashed] (0,1) -- (2,0);
    \draw (2,0) -- (6,0);
    \draw (2,0) -- (6,-2);
    \draw (-1,0) -- (-6,0);
    \draw (-1,0) -- (-3,-2);
    \draw[dotted] (-6,-1) -- (6,-1);
    \foreach \x in {-6,...,6}{
        \foreach \y in {-2,...,1}{
            \node[draw,circle,inner sep=0.1pt,fill] at (\x,\y) {};
        }
    }
    \node[anchor=south] at (-1,0) {\(x_2\)};
    \node[anchor=south] at (2,0) {\(x_1\)};
    \draw[fill] (-1,0) circle (1.5pt);
    \draw[fill] (2,0) circle (1.5pt);
    \foreach \x in {-1,...,3}{
        \draw[fill=white] (\x,-1) circle (1.7pt);
    }
    \foreach \x in {-6,-4,...,6}{
        \foreach \y in {-2,0}{
            \node[cross=2.5pt] at (\x,\y) {};
        }
    }
\end{tikzpicture}
\caption{The points which Algorithm~\ref{alg:growing} adds to \(P=\conv((0,0),(1/2,0),(0,1/2))\) on the hyperplane adjacent to \(\conv((0,0),(1/2,0))\).
Dots denote points of \(\frac12\ZZ^2\) and crosses denote points of \(\ZZ^2\). The points which we can add are denoted by circles and \(x_1=(1,0)\) and \(x_2=(-\frac12,0)\) are points of the hyperplane \(y=0\) which we may not include in the new polygon.}
\label{fig:growing_eg}
\end{figure}
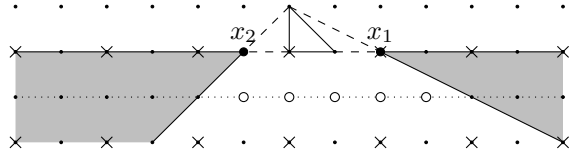

For the algorithm to terminate it must also have a finite number of polygons to classify.
In particular, there must be a finite number of finitely growable polygons of a given size and denominator.
\begin{proposition}\label{prop:hip_fin}
Let \(r \in \ZZ_{>0}\) and \(k \in \ZZ_{\geq0}\), then there are finitely many finitely growable polygons with denominator \(r\) and size \(k\).
\end{proposition}
\begin{proof}
By Proposition~\ref{prop:hollow} it suffices to show that there are finitely many such polygons with interior lattice points.
By \cite[Theorem 1]{LagariasZiegler} the volume of a polygon containing exactly \(l \geq 1\) points of \(r \ZZ^2\) in its interior is at most \(lr^2(7(lr+1))^{16}\).
Therefore, the volume of a polygon containing at most \(k-3\) interior points is also bounded by some number \(V\).
By \cite[Theorem 2]{LagariasZiegler} this bound means that the finitely growable polygons with denominator \(r\) and size \(k\) are equivalent to polygons contained in a lattice square of side length at most \(4V\).
Therefore, there are finitely many of them.
\end{proof}

Note that a sufficiently small infinitely growable polygon may also be a subset of a maximal polygon.
Our growing algorithm cannot exclude infinitely growable polygons, since some minimal polygons are infinitely growable.
We still need to classify a finite set so we bound the collection of infinitely growable polygons which we include using the following result.
\begin{proposition}\label{prop:colinear_points_bound}
Let \(P\) be a rational polygon with denominator \(r\) and size \(k\).
Let \(H\) be a hyperplane defined by \(u \cdot x = \widetilde{h}\) for some non-integral \(\widetilde{h} \in \frac1r\ZZ\).
Let \(h\) be the minimum of \((r\widetilde{h} \mod r)\) and \((-r\widetilde{h} \mod r)\).
If
\begin{equation*}
    \mleft|P \cap H \cap \tfrac1r\ZZ^2\mright| \geq r(r-h+1)(k+1)
\end{equation*}
then \(P\) is not a subset of a maximal polygon.
\end{proposition}
\begin{proof}
    We may assume by an affine map that \(u=(1,0)\) and that \(0 <\widetilde{h}<1\).
    Notice then that \(h=r\widetilde{h}\) or \(r-r\widetilde{h}\).
    The line segment \(P \cap H\) contains \(a\) points from \(\frac1r\ZZ^2\) for some integer \(a \geq r(r+1-h)(k+1)\) so the length of this line segment is at least \(\frac{a-1}r\).
    
    Assume towards a contradiction that there exists a maximal rational polygon \(Q\) with denominator \(r\) and size \(k\) containing \(P\).
    Since \(Q\) is not infinitely growable it contains a point \(v_1=(x_1,y_1)\) in \(\frac1r\ZZ^2\) outside of the strip \([0,1] \times \RR\).
    By a reflection in the \(y\)-axis followed by a horizontal translation we may assume that \(x_1\geq 1+\frac1r\).
    We claim that there exists a point of \(\frac1r\ZZ^2\) which is contained in the line segment \(Q\cap\{x=1+\frac1r\}\).

    Suppose that \(x_1\geq1+\frac2r\) and let \(T_1\) be the triangle \(\conv(P\cap H, v_1)\) which is contained in \(Q\).
    Applying the intercept theorem on the two triangles \(T_1\) and \(T_1\cap\{x\ge1+\frac1r\}\) yields that the length of the line segment \(T_1 \cap \{x=1+\frac1r\}\) is at least
    \begin{equation} \label{eq:length_line_seg_outside}
    \frac{a-1}{r}\cdot \frac{rx_1-r-1}{rx_1-r\widetilde{h}}.
    \end{equation}
    We claim that this lower bound is at least \(\frac1r\) which then guarantees that the line segment \(Q\cap\{x=1+\frac1r\}\) contains a rational point in \(\frac1r\ZZ^2\).
    Notice that the partial derivative of \eqref{eq:length_line_seg_outside} with respect to \(x_1\) is positive:
    \[
    \frac{\partial}{\partial x_1}\frac{a-1}{r}\cdot \frac{rx_1-r-1}{rx_1-h'} = \frac{(a-1)(r+1-h')}{r^2(rx_1-h')^2}>0
    \]
    so \eqref{eq:length_line_seg_outside} is increasing.
    Since \(x_1\ge 1+\frac2r\), a lower bound for the fraction \eqref{eq:length_line_seg_outside} is given by substituting in \(x_1=1+\frac2r\), that is,
    \[
    \frac{a-1}{r}\cdot \frac{rx_1-r-1}{rx_1-h'}\ge\frac{a-1}{r(r+2-r\widetilde{h})}.
    \]
    Thus it suffices to show that \((a-1)/(r(r+2-r\widetilde{h}))\ge\frac1r\).
    By definition \(a \geq r(r+1-h)(k+1)\) and \(k \geq 0\) so it further suffices to show that
    \[
    \frac{r(r+1-h)-1}{r+2-r\widetilde{h}}\ge 1.
    \]
    This can be verified using the facts that \(r\widetilde{h} \in \{h,r-h\}\) and \(h\le\frac r2\).
    This proves that \(Q\) contains a point of \(\frac1r\ZZ^2\) with \(x\)-coordinate \(1+\frac1r\).

    From now on, assume that \(v_1\) has \(x\)-coordinate \(x_1 = 1+\frac1r\).
    Let us denote the \(a\) consecutive points of \(P\cap H\cap \frac1r\ZZ^2\) by \(p_1, \dots, p_a\).
    Let \(q_i\) be the rational intersection points of the line segments \(\conv(v_1,p_i)\) with \(S = Q \cap \{x=1\}\).
    By computing the intersection of these line segments in general we see that the \(y\)-coordinate of \(q_i\) is in \(\frac{1}{r(r+1-r\widetilde{h})}\ZZ\).
    Thus, \(S\) contains at least \(r(r+1-h)(k+1)\) points of \(\frac{1}{r(r+1-r\widetilde{h})}\ZZ^2\).
    At least \(k+1\) of these must be integral points which contradicts the size of \(Q\).
\end{proof}

We make the growing algorithm rigorous in Algorithm~\ref{alg:growing}. 
We have shown that it starts with finitely many polygons and at each growing step adds finitely many additional polygons with strictly larger volume.
Additionally all polygons it produces are members of a finite set.
Therefore, the algorithm terminates and classifies all finitely growable polygons of a given size and denominator.

Three final adjustments can be made to make the process more efficient.
First notice that affine equivalent polygons contain the same number of points in \(\frac1r\ZZ^2\).
Also, our algorithm grows polygons by exactly one point of \(\frac1r\ZZ^2\) at a time.
Therefore, we can stratify the growing algorithm by the \(\frac1r\ZZ^2\)-size of the polygons.
That is, grow all polygons with \(\frac1r\ZZ^2\)-size \(k\) then \(k+1\) and so on.
The benefit of this is that once all polygons of a given \(\frac1r\ZZ^2\)-size have been classified, and all duplicates and infinitely growable polygons have been removed, they can be stored and deleted from working memory without any danger that we will classify them again in later iterations.

The next adjustment relates to removing infinitely growable polygons from the final classification.
Notice that a finitely growable polygon can never be grown into an infinitely growable polygon.
Thus, once we have checked and found that a polygon is finitely growable we need not check whether any polygon grown from it is also finitely growable.
To avoid this we distinguish between polygons which are infinitely growable and finitely growable throughout the main loop.

Finally, we will show in Section~\ref{sec:minimal_polygons} that a polygon with denominator \(r \geq 1\) and size \(k \geq 1\) contains a unique minimal polygon and two such polygons can only be equivalent if their minimal polygons are equivalent (see Remark~\ref{rem:only_one_min}).
As a result, when the size is at least 1 the algorithm grows each minimal polygon independently one after the other rather than simultaneously.
This reduces the number of polygons which need to be stored in working memory.

\begin{algorithm}
\DontPrintSemicolon

\SetKwFunction{MinPoly}{MinPoly}
\SetKwFunction{Final}{Final}
\SetKwFunction{ToGrowInf}{ToGrowInf}
\SetKwFunction{ToGrowFin}{ToGrowFin}
\SetKwFunction{ToGrowNextInf}{ToGrowNextInf}
\SetKwFunction{ToGrowNextFin}{ToGrowNextFin}
\SetKwFunction{size}{size}
\SetKwFunction{NewInf}{NewInf}
\SetKwFunction{NewFin}{NewFin}
\SetKwFunction{Grow}{Grow}
\SetKwFunction{IsFin}{IsFin}
\SetKwFor{Function}{function}{}{}

\KwData{The set \(\MinPoly\) containing minimal polygons with denominator \(r\) and size \(k\).}
\KwResult{The set \(\Final\) containing all finitely growable polygons with denominator \(r\) and size \(k\).}
\tcc{A function which grows a given polytope by exactly one point of \(\frac1r\ZZ^2\) in every way possible without increasing the size or adding too many colinear points}
\Function{\Grow{\(P\),\(r\),\(r\)-size,\IsFin}}{
    \(\NewInf \longleftarrow \emptyset\)\;
    \(\NewFin \longleftarrow \emptyset\)\;
    \For{\(p \in\{p \in  H \cap \frac1r\ZZ^2, p \notin \pen(P,x \in \frac1r\ZZ^2\setminus H)\), \(H\) a hyperplane adjacent to a facet of \(P\)\(\}\)}{
        \(Q \longleftarrow \conv(P,p)\)\;
        \If{\(Q\) has correct size, \(r\)-size and not too many colinear points according to Prop~\ref{prop:colinear_points_bound}}{
            \uIf{\IsFin or \(Q\) is finitely growable}{
            \(\NewFin \longleftarrow \NewFin \cup \{Q\}\)\;
            }
            \Else{
            \(\NewInf \longleftarrow \NewInf \cup \{Q\}\)\;
            }
        }
    }
    \If{\IsFin}{
        \textbf{return} \NewFin\;
    }
    \textbf{return} \NewInf, \NewFin\;
}
\(\Final \longleftarrow \emptyset\)\;
\tcc{Main growing loop, grows each minimal polygon iteratively until any point which can be added to a polygon increases it's size}
\For{\(P_{min} \in \MinPoly\)}{
    \uIf{\(k=0\)}{
        \(\ToGrowInf \longleftarrow \MinPoly\)\;
        \(\ToGrowFin \longleftarrow \emptyset\)\;
    }
    \uElseIf{\(P_{min}\) is infinitely growable}{
        \(\ToGrowInf \longleftarrow \{P_{min}\}\)\;
        \(\ToGrowFin \longleftarrow \emptyset\)\;
    }
    \Else{
        \(\ToGrowInf \longleftarrow \emptyset\)\;
        \(\ToGrowFin \longleftarrow \{P_{min}\}\)\;
    }
    \(r\)-size \(\longleftarrow |P \cap \frac1r\ZZ^2|\)\;
    \Repeat{\(\ToGrowFin = \emptyset\) and \(\ToGrowInf=\emptyset\)}{
        \(r\)-size \(\longleftarrow r\)-size\(+1\)\;
        \(\ToGrowNextInf \longleftarrow \emptyset\)\;
        \(\ToGrowNextFin \longleftarrow \emptyset\)\;
        \For{\(P \in \ToGrowInf\)}{
            \(\NewInf,\NewFin \longleftarrow \Grow(P,r,r\text{-size},\text{false})\)\;
            \(\ToGrowNextInf \longleftarrow \ToGrowNextInf \cup \NewInf\)\;
            \(\ToGrowNextFin \longleftarrow \ToGrowNextFin \cup \NewFin\)\;
        }
        \For{\(P \in \ToGrowFin\)}{
            \(\NewFin \longleftarrow \Grow(P,r,r\text{-size},\text{true})\)\;
            \(\ToGrowNextFin \longleftarrow \ToGrowNextFin \cup \NewFin\)\;
        }
        \(\Final \longleftarrow \Final \cup \ToGrowFin\)\;
        \(\ToGrowInf \longleftarrow \ToGrowNextInf\)\;
        \(\ToGrowFin \longleftarrow \ToGrowNextFin\)\;
    }
    \If{\(k=0\)}{
        \textbf{break}\;
    }
}
\caption{Growing algorithm for polygons with denominator \(r\) and size \(k\). All sets are considered modulo affine equivalence.}
\label{alg:growing}
\end{algorithm}

\section{Minimal polygons}
\label{sec:minimal_polygons}

In this section we classify the minimal polygons with denominator \(r > 1\) and size \(k\geq 0\).
Those with size \(0\) are a separate case which we classify first.

\begin{proposition}\label{prop:min_size_zero}
    The minimal polygons of size zero with denominator \(r\) are the size zero triangles of the form \(\frac1r(\Delta+v)\) where \(\Delta \coloneqq \conv((0,0),(1,0),(0,1))\) and \(v\) is a lattice point in the square \([0,r-1]^2\).
\end{proposition}
Note that not all triangles of the form \(\frac1r(\Delta+v)\) have size zero so the condition `size zero' in the proposition places restrictions on \(v\).
\begin{proof}
First we can see that these triangles are minimal since removing any of their vertices makes them one-dimensional.

Let \(P\) be a minimal polygon of size zero with denominator \(r\).
We can find a triangulation of \(P\) into triangles containing exactly three points of \(\frac1r\ZZ^2\).
Let \(T\) be one of these triangles.
Using Pick's theorem we may assume that \(T\) is some affine transformation of \(\frac1r\Delta\).
By an affine map, we may assume that \(T\) is \(\frac1r(\Delta+v)\) for some lattice point \(v\) in \([0,r-1]^2\).
Any vertex of \(P\) not contained in \(T\) can be removed without changing the size or dimension so \(P=T\).
\end{proof}

\begin{proposition}\label{prop:minimal_size_zero2}
    There is a bijection between the minimal polygons of size zero with denominator \(r\) and the set \(A\) of triples \((a_1,a_2,a_3) \in \ZZ^3_{\geq 0}\) with \(a_1\leq a_2 \leq a_3 \leq r-1\) satisfying one of the following conditions:
    \begin{enumerate}
        \item \(a_1+a_2+a_3=r-1\) and \(0 < a_1+a_2\) 
        \item \(a_1+a_2+a_3=2r-1\) and \(r \leq a_1+a_2\) 
    \end{enumerate}
\end{proposition}
\begin{proof}
By Proposition~\ref{prop:min_size_zero} it suffices to show that there is a bijection between \(A\) and the set \(S\) of affine equivalence classes \([\frac1r(\Delta+v)]\) where \(v\) is a lattice point such that \(\frac1r(\Delta+v)\) contains no lattice points.

Let \(T \coloneqq \frac1r(\Delta+v)\).
We can write \(T\) as the intersection of three half-spaces \(H_i \coloneqq \{v \in \RR^2: u_i\cdot v \geq \widetilde{h_i}\}\) for \(i=1,2,3\) where each \(\widetilde{h_i} \in \frac1r\ZZ\).
We define \(h_i = (r\widetilde{h_i} \mod r)\) and, up to relabeling, we may assume that \(h_1 \leq h_2 \leq h_3\).
By a unimodular map we may assume that \(u_1=(1,0)\) and \(u_2=(0,1)\) and by an integral translation we may assume that \(r\widetilde{h_1}=h_1\) and \(r\widetilde{h_2}=h_2\).
This means that \(T\) is equivalent to \(\frac1r((h_1,h_2)+\Delta)\) and so we can calculate \(h_3\) in terms of \(h_1\) and \(h_2\).
That is, \(h_3=r-1-h_1-h_2\) when \(h_1+h_2 <r\) and \(h_3=2r-1-h_1-h_2\) otherwise.
Since \(T\) contains no lattice points \(h_1+h_2\) is always greater than \(0\).
This shows that \((h_1,h_2,h_3) \in A\).

We claim that the map from \(S\) to \(A\) taking the equivalence class \([T]\) to \((h_1,h_2,h_3)\) is the desired bijection.
This is a well-defined map as the triple \((h_1,h_2,h_3)\) is invariant under affine maps.
It remains to prove that it is injective and surjective.
As shown above, if \([T] \mapsto (h_1,h_2,h_3)\) then \(T\) is equivalent to the triangle \(\frac1r((h_1,h_2)+\Delta)\) from which we can deduce injectivity.
Surjectivity comes from the fact that for any \((a_1,a_2,a_3) \in A\) we have \([\frac1r((a_1,a_2)+\Delta)] \mapsto (a_1,a_2,a_3)\).
\end{proof}

\begin{proposition}\label{prop:class_of_minimal}
    The minimal polygons with denominator \(r\) and size \(k > 1\) are exactly the integral polygons of size \(k\) and the triangle
    \[
    T_{r,k}\coloneqq \conv\left((0,0),(k-1,0),\left(0,\tfrac{1}{r}\right)\right)
    \]
    and the only minimal polygon with denominator \(r\) and size \(k=1\) is the triangle 
    \[
    T_{r,1}\coloneqq \conv\left((0,0),\left(\tfrac{1}{r},0\right),\left(0,\tfrac{1}{r}\right)\right).
    \]
\end{proposition}
\begin{proof}
    The integral polygons of size \(k\) are minimal since if we removed any of their vertices the result would have size \(k-1\).
    The triangle \(T_{r,k}\) is minimal for all \(k \geq 1\), since if we removed an integral vertex from it the result would have size \(k-1\) and if we removed a rational vertex the result would be a line.
    
    Suppose \(P\) is a minimal polygon with denominator \(r\) and size \(k > 1\).
    Let \(Q\) be the convex hull of the lattice points in \(P\).
    This is either a lattice polygon of size \(k\) or a line segment of lattice length \(k-1\).
    If \(Q\) is an lattice polygon of size \(k\) then any vertex of \(P\) not contained in \(Q\) can be removed without making the result a smaller size or dimension, therefore \(P=Q\) is a lattice polygon of size \(k\).
    If instead \(Q\) is a line segment of lattice length \(k-1\), then by an affine map we may assume that \(P\) contains the line segment between \((0,0)\) and \((k-1,0)\).
    By a reflection we may assume that \(P\) contains a \(\frac{1}{r}\)-integral point with positive \(y\)-coordinate.
    By Lemma~\ref{lem:Picks_point_on_line} \(P\) contains a point with \(y\)-coordinate \(\frac{1}{r}\) which by a shear we may assume is \((0,\frac{1}{r})\).
    Any point of \(P\) not contained in \(\conv((0,0), (k-1,0), (0,\frac{1}{r}))\) can be removed without making the result a smaller size or dimension, therefore \(P=\conv((0,0), (k-1,0), (0,\frac{1}{r}))\).
    
    If \(P\) is a minimal polygon with denominator \(r\) and size \(1\) an affine transformation allows us to assume it contains the points \((0,0)\) and \((\frac1r,0)\).
    A similar argument to above shows that \(P\) is equivalent to \(T_{r,1}\).
\end{proof}

\begin{remark}\label{rem:only_one_min}
Notice that the minimal polygon contained in a rational polygon of size at least 1 is entirely determined by the convex hull of the lattice points it contains.
As a result, two such polygons can only be equivalent if they can be grown from the same minimal polygon.
\end{remark}

To compute the minimal polygons, it remains to classify the lattice polygons of a given size.
Algorithms to do so have been presented elsewhere, for example in \cite{Class_conv_poly} and \cite{Latt_3_topes_few_pts}.
We use a very similar growing algorithm to the one described in Section~\ref{sec:growing_alg}, which adds points on hyperplanes adjacent to the facets of a polygon which is the approach mentioned in \cite{Latt_3_topes_few_pts}.

\section{Ehrhart Theory of Rational Polygons}
\label{sec:ehrhart}

As mentioned previously, the Ehrhart polynomial of a lattice polygon determines and is determined by the number of boundary and interior points of that polygon.
Thus, the following result gives a complete classification of the Ehrhart polynomials of lattice polygons.
\begin{theorem}[\cite{Scott,Haase_Schicho}]\label{thm:Scotts}
    Integers \(b\) and \(i\) are the number of boundary and interior points of a lattice polygon if and only if \(b \geq 3\), \(i \geq 0\) and one of the following holds
    \begin{itemize}
        \item \(i=0\),
        \item \(i=1\) and \(b = 9\) or
        \item \(i \geq 1 \) and \(b \leq 2i+6\).
    \end{itemize}
\end{theorem}
The only polygon with one interior point and 9 boundary points is the triangle \(\conv((0,0),(3,0),(0,3))\).

Towards a generalisation of this result for rational polygons we first show that the Ehrhart polynomial of a rational polygon can be written in terms of numbers of boundary and interior points of polygons.

\begin{proposition}\label{prop:rat_Ehr_poly}
Let \(P\) be a rational polygon with Ehrhart quasi-polynomial
\[
\ehr_P(n) = a_{2,i}n^2 + a_{1,i}n + a_{0,i}\quad \text{when } n \equiv i \mod r
\]
for some positive integer \(r\) and \(i \in \{0,1,\dots, r-1\}\). Then 
\begin{align*}
    a_{2,i} & = \tfrac12\Vol(P)\\
    a_{1,i} & = \tfrac1r\left(|iP\cap \ZZ^2| - |(r-i)P^\circ \cap \ZZ^2| - \tfrac{r(2i - r)}{2}\Vol(P)\right)\\
    a_{2,i}i^2+a_{1,i}i+a_{0,i} &= |iP \cap \ZZ^2|
\end{align*}
for all \(i\) where \(\Vol(P)\) denotes the normalised volume of \(P\).
\end{proposition}
\begin{remark}
    The case \(r=2\) was presented before in \cite[Lemma~3.3]{Herrmann}.
    We use the same method to extend this to the general quasi-period case.
\end{remark}
The proof depends on the following theorem.
\begin{theorem}[Ehrhart--Macdonald reciprocity \cite{Ehr_Mac_recip}]\label{thm:ehr_mac_recip}
    Let \(P\subseteq \RR^d\) be a rational polytope and let \(n\) be a positive integer.
    Then,
    \[
    (-1)^{\dim(P)}\ehr_P(-n) = |nP^\circ \cap \ZZ^d|
    \]
    where \(P^\circ\) is the relative interior of \(P\).
\end{theorem}
\begin{proof}[Proof of Proposition~\ref{prop:rat_Ehr_poly}]
The fact that \(a_{2,i}=\frac12\Vol(P)\) is a known result coming from considering the limit as \(n\) tends to infinity.
By the definition of \(\ehr_P(n)\), for \(i=0,1,\dots,r-1\)
\[
a_{2,i}i^2+a_{1,i}i+a_{0,i} = |iP \cap \ZZ^2|
\]
and by Ehrhart-Macdonald reciprocity
\[
a_{2,i}(i-r)^2+a_{1,i}(i-r)+a_{0,i} = |(r-i)P^\circ \cap \ZZ^2|.
\]
Taking the difference of these two equations and simplifying shows that 
\[
a_{1,i} = \frac1r\left(|iP\cap \ZZ^2| - |(r-i)P^\circ \cap \ZZ^2| - \frac{r(2i - r)}{2}\Vol(P)\right).
\]
\end{proof}
Since \(rP\) is a lattice polygon we can use Pick's theorem to write the volume of \(P\) in terms of the number of boundary and interior points of \(rP\).
Similarly to lattice polygons, this shows that the Ehrhart polynomial of a denominator \(r\) polygon is completely encoded by the number of boundary and interior points of \(P\), \(2P\), \dots, \((r-1)P\) and \(rP\).
Therefore, it suffices to study such numbers to understand the Ehrhart theory of rational polygons.

\section{Ehrhart Theory of Denominator 2 Polygons}

In this section we study the number of interior and boundary points in \(P\) and \(2P\) where \(P\) is a polygon with denominator 2.
Recall, Theorem~\ref{thm:hip_finlap_main} states that all but finitely many such polygons \(P\) must satisfy one of four collections of linear bounds on the number of interior and boundary points of \(P\) and \(2P\).
We prove this in Propositions~\ref{prop:inf_half_Ehr_1}, \ref{prop:inf_half_Ehr_2}, \ref{prop:ehr_diag_bound}, \ref{prop:ehr_h_v_bounds} and Table~\ref{tab:zero_int_pts_ehr}.

\subsection{Polygons With Zero Interior Points}
\subsubsection{Infinitely growable polygons}
The width of a polytope is an affine invariant which measures the number of integral hyperplanes it passes through.
The \emph{width of \(P\) with respect to \(u\)}, where \(u\) is a dual lattice vector, is defined to be
\[
\width_u(P) \coloneqq \max_{x \in P} \{u \cdot x\} - \min_{x \in P} \{u \cdot x\}.
\]
The \emph{(first) width} of \(P\) is defined to be 
\[
\width(P) = \min_{u \neq 0} \{\width_u(P)\}.
\]
We showed in Section~\ref{sec:infty-grow-hip} that infinitely growable rational polygons are those which are equivalent to a subset of the strip \([0,1] \times \RR\).
A consequence of this is that the width of an infinitely growable polygon is at most 1.
This helps us prove the following.
\begin{proposition} \label{prop:inf_half_Ehr_1}
    Let \(P\) be an infinitely growable, denominator 2 polygon.
    Let \(i_1\), \(b_1\), \(i_2\) and \(b_2\) be the number of interior and boundary points of \(P\) and \(2P\) respectively.
    Then \(i_1=0\) 
    and the remaining variables satisfy one of the following conditions:
    \begin{enumerate}
        \item \(i_2=0\) and \(b_2 \geq \max(3,2b_1)\),
        \item \(i_2>0\), \(b_1=0\) and \(b_2=4\) or
        \item \(i_2,b_1>0\), \(\max(3,2b_1) \leq b_2 \leq 2b_1+4\) and \(b_2 \leq 2i_2+6\).
    \end{enumerate}
\end{proposition}
See Figure~\ref{fig:0_int} for plots of pairs \((b_2,i_2)\) when \(b_1,i_2\) and \(b_2\) are small.

\begin{figure}[b]
\centering
\subfloat[\(P\) has 0 boundary points]{
    \begin{tikzpicture}
    \begin{axis}[width=0.45\textwidth, axis lines=left, xmin=0, ymin=0, xlabel = {\(|\partial 2P \cap \ZZ^2|\)}, ylabel = {\(|2P^\circ \cap \ZZ^2|\)}]
    \addplot[only marks, mark=x, mark size=3.2pt]
    table{data/0_boundary_0_interior_inf.dat};
    \addplot[only marks, mark=*, mark size=1.5pt]
    table{data/0_boundary_0_interior_fin.dat};
    \addplot[densely dotted,color=black] coordinates {(4, 0) (4, 10)};
    \addplot[densely dotted,color=black] coordinates {(6, 0) (10,2)};
    \end{axis}
    \end{tikzpicture}
}
\subfloat[\(P\) has  1 boundary point]{
    \begin{tikzpicture}
    \begin{axis}[width=0.45\textwidth, axis lines=left, xmin=0, ymin=0, xlabel = {\(|\partial 2P \cap \ZZ^2|\)}, ylabel = {\(|2P^\circ \cap \ZZ^2|\)}]
    \addplot[only marks, mark=x, mark size=3.2pt]
    table{data/1_boundary_0_interior_inf.dat};
    \addplot[only marks, mark=*, mark size=1.5pt]
    table{data/1_boundary_0_interior_fin.dat};
    \addplot[densely dotted,color=black] coordinates {(3, 0) (3, 10)};
    \addplot[densely dotted,color=black] coordinates {(6, 0) (6, 10)};
    \addplot[densely dotted,color=black] coordinates {(6, 0) (10,2)};
    \end{axis}
    \end{tikzpicture}
}\quad
\subfloat[\(P\) has 2 boundary points]{
    \begin{tikzpicture}
    \begin{axis}[width=0.45\textwidth, axis lines=left, xmin=2, ymin=0, xlabel = {\(|\partial 2P \cap \ZZ^2|\)}, ylabel = {\(|2P^\circ \cap \ZZ^2|\)}]
    \addplot[only marks, mark=x, mark size=3.2pt]
    table{data/2_boundary_0_interior_inf.dat};
    \addplot[only marks, mark=*, mark size=1.5pt]
    table{data/2_boundary_0_interior_fin.dat};
    \addplot[densely dotted,color=black] coordinates {(4, 0) (4, 10)};
    \addplot[densely dotted,color=black] coordinates {(8, 0) (8, 10)};
    \addplot[densely dotted,color=black] coordinates {(6, 0) (12,3)};
    \end{axis}
    \end{tikzpicture}
}
\subfloat[\(P\) has 3 boundary points]{
    \begin{tikzpicture}
    \begin{axis}[width=0.45\textwidth, axis lines=left, xmin=4, ymin=0, xlabel = {\(|\partial 2P \cap \ZZ^2|\)}, ylabel = {\(|2P^\circ \cap \ZZ^2|\)}]
    \addplot[only marks, mark=x, mark size=3.2pt]
    table{data/3_boundary_0_interior_inf.dat};
    \addplot[only marks, mark=*, mark size=1.5pt]
    table{data/3_boundary_0_interior_fin.dat};
    \addplot[densely dotted,color=black] coordinates {(6, 0) (6, 10)};
    \addplot[densely dotted,color=black] coordinates {(10, 0) (10, 10)};
    \addplot[densely dotted,color=black] coordinates {(6, 0) (14,4)};
    \end{axis}
    \end{tikzpicture}
}
\caption{Plots of \((b(2P),i(2P))\) for denominator 2 polygons of size up to 6 with zero interior points. Crosses are realised by infinitely growable polygons and dots by finitely growable polygons. Dotted lines denote the bounds described in Proposition~\ref{prop:inf_half_Ehr_1}. Continued on next page.}
\label{fig:0_int}
\end{figure}
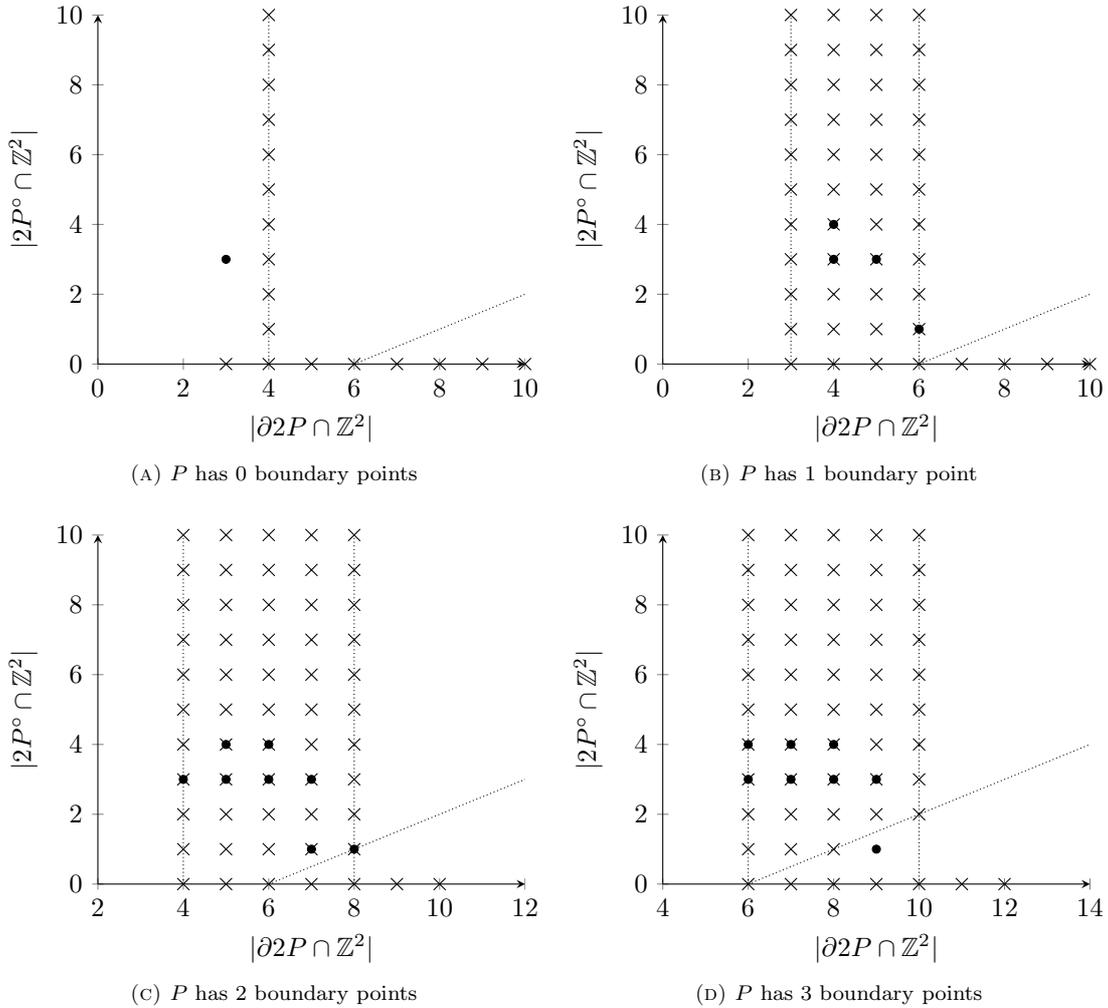

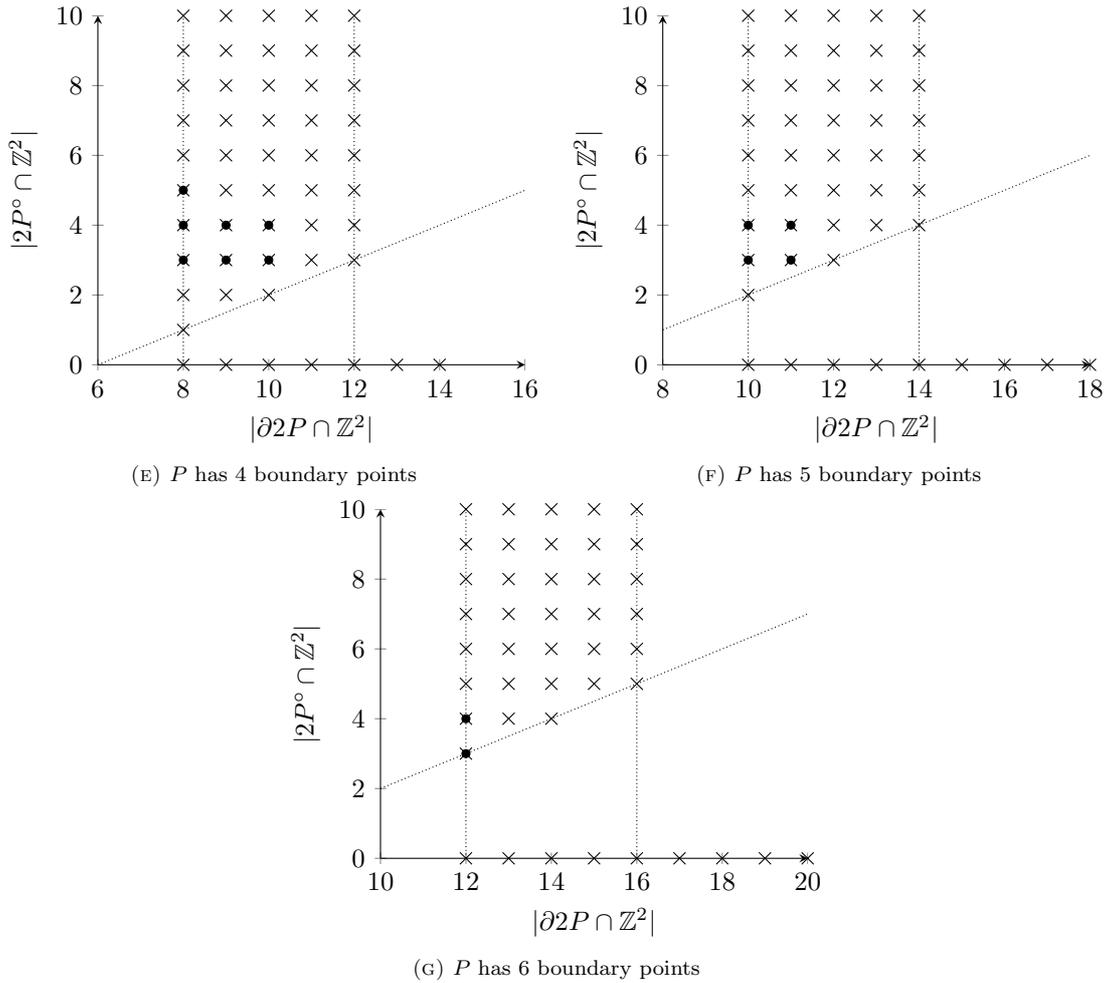
\begin{figure}[ht]
\centering
\ContinuedFloat
\subfloat[\(P\) has 4 boundary points]{
    \begin{tikzpicture}
    \begin{axis}[width=0.45\textwidth, axis lines=left, xmin=6, ymin=0, xlabel = {\(|\partial 2P \cap \ZZ^2|\)}, ylabel = {\(|2P^\circ \cap \ZZ^2|\)}]
    \addplot[only marks, mark=x, mark size=3.2pt]
    table{data/4_boundary_0_interior_inf.dat};
    \addplot[only marks, mark=*, mark size=1.5pt]
    table{data/4_boundary_0_interior_fin.dat};
    \addplot[densely dotted,color=black] coordinates {(8, 0) (8, 10)};
    \addplot[densely dotted,color=black] coordinates {(12, 0) (12, 10)};
    \addplot[densely dotted,color=black] coordinates {(6, 0) (16,5)};
    \end{axis}
    \end{tikzpicture}
}
\subfloat[\(P\) has 5 boundary points]{
    \begin{tikzpicture}
    \begin{axis}[width=0.45\textwidth, axis lines=left, xmin=8, ymin=0, xlabel = {\(|\partial 2P \cap \ZZ^2|\)}, ylabel = {\(|2P^\circ \cap \ZZ^2|\)}]
    \addplot[only marks, mark=x, mark size=3.2pt]
    table{data/5_boundary_0_interior_inf.dat};
    \addplot[only marks, mark=*, mark size=1.5pt]
    table{data/5_boundary_0_interior_fin.dat};
    \addplot[densely dotted,color=black] coordinates {(10, 0) (10, 10)};
    \addplot[densely dotted,color=black] coordinates {(14, 0) (14, 10)};
    \addplot[densely dotted,color=black] coordinates {(8, 1) (18,6)};
    \end{axis}
    \end{tikzpicture}
}\quad
\subfloat[\(P\) has 6 boundary points]{
    \begin{tikzpicture}
    \begin{axis}[width=0.45\textwidth, axis lines=left, xmin=10, ymin=0, xlabel = {\(|\partial 2P \cap \ZZ^2|\)}, ylabel = {\(|2P^\circ \cap \ZZ^2|\)}]
    \addplot[only marks, mark=x, mark size=3.2pt]
    table{data/6_boundary_0_interior_inf.dat};
    \addplot[only marks, mark=*, mark size=1.5pt]
    table{data/6_boundary_0_interior_fin.dat};
    \addplot[densely dotted,color=black] coordinates {(12, 0) (12, 10)};
    \addplot[densely dotted,color=black] coordinates {(16, 0) (16, 10)};
    \addplot[densely dotted,color=black] coordinates {(10, 2) (20,7)};
    \end{axis}
    \end{tikzpicture}
}
\caption{Plots of \((b(2P),i(2P))\) for denominator 2 polygons of size up to 6 with zero interior points continued.}
\end{figure}

\begin{proof}
    The fact that \(P\) is infinitely growable shows that \(i_1=0\).
    It is immediate that exactly one of the following three conditions holds
    \begin{enumerate}[(a)]
        \item \(i_2=0\),
        \item \(i_2>0\) and \(b_1=0\) or
        \item \(i_2>0\) and \(b_1>0\).
    \end{enumerate}
    It suffices to show that, in each of these cases, the corresponding condition of the proposition holds.
    
    The integers \(b_2\) and \(i_2\) are the number of boundary and interior points of a lattice polygon so they must satisfy Theorem~\ref{thm:Scotts}.
    The exceptional case where \(b_2=9\) and \(i_2=1\) never occurs since the only possible \(2P\) in this case has width 3, contradicting the fact that \(P\) is infinitely growable.
    Therefore, \(b_2 \geq 3\) and either \(i_2=0\) or \(b_2 \leq 2i_2+6\).
    
    For each boundary point \(v\) of \(P\) we know that \(2v\) is a boundary point of \(2P\).
    For each pair of adjacent boundary points of \(P\) either they are connected by part of a single edge of \(P\) or there is a half-integral vertex between them on the boundary of \(P\).
    In either case, there is a half-integral point on the boundary of \(P\) between each such pair so \(b_2 \geq 2b_1\).
    Therefore, when \(i_2=0\) condition (1) holds
    
    From now on we assume that \(i_2 >0\).
    We may assume that \(P\) is a subset of \([0,1] \times \RR\) and consider the number of boundary points \(2P\) has on each of the hyperplanes \(x=0,1\) and \(2\).
    Since \(i_2>0\), \(2P\) has interior points on the line \(x=1\), so \(2P\) can have at most two boundary points on this line.
    On the lines \(x=0\) and \(x=2\) together \(2P\) can have at most \(2b_1+2\) boundary points so \(b_2 \leq 2b_1+4\).
    Therefore, when \(i_2\) and \(b_1\) are positive, condition (3) holds.
    
    If \(b_1=0\) and \(i_2>0\) then \(2P\) contains interior lattice points in the line \(x=1\) but contains at most one point in each of the lines \(x=0\) and \(2\).
    Therefore, \(2P\) contains exactly one boundary point on each of the lines \(x=0\) and \(2\).
    Call these \(v_0\) and \(v_2\).
    The only possible remaining vertices of \(2P\) are in the line \(x=1\).
    If \(2P\) has two vertices in the line \(x=1\) then we have shown that \(b_2\geq 4\).
    Otherwise, the line segment from \(v_0\) to \(v_2\) is an edge of \(2P\).
    It contains a lattice point at its midpoint since it is affine equivalent to \(\conv((1,0),(1,2))\).
    Therefore, \(2P\) has two boundary points on the line \(x=1\) and \(b_2\geq 4\).
    Therefore, when \(i_2\) is positive and \(b_1=0\), condition (2) holds.
\end{proof}

\begin{proposition}\label{prop:inf_half_Ehr_2}
    Let \(b_1,b_2\) and \(i_2\) be non-negative integers satisfying one of the conditions in Proposition~\ref{prop:inf_half_Ehr_1}.
    Then there is an infinitely growable, denominator 2 polygon \(P\) such that \(P\) has \(b_1\) boundary points and \(2P\) has \(b_2\) boundary points and \(i_2\) interior points.
\end{proposition}
\begin{proof}
    We consider tuples \((b_1,b_2,i_2)\) of non-negative integers which satisfy the conditions of Proposition~\ref{prop:inf_half_Ehr_1}.
    We say that a polygon \(P\) realises such a tuple if it is infinitely growable, has \(b_1\) boundary points and \(2P\) has \(b_2\) boundary points and \(i_2\) interior points.
    We give examples of infinite families of denominator 2 polygons realising all of these tuples.
    
    The first condition is satisfied by tuples \((0,b_2,0)\) where \(b_2 \geq 3\), \((1,b_2,0)\) where \(b_2 \geq 3\) and \((b_1,b_2,0)\) where \(b_2 \geq 2b_1 \geq 4\).
    These are realised by the polygons
    \begin{itemize}
        \item \(\conv((0,\frac12),(\frac12,0),(\frac12,\frac{b_2-2}{2})\) if \(b_1=0\), and 
        \item \(\conv((0,0),(0,b_1-1),(\frac12,0),(\frac12,\frac{b_2-2b_1}{2})\) if \(b_1>0\).
    \end{itemize}
    
    The second condition is satisfied by tuples \((0,4,i_2)\) where \(i_2>0\).
    These are realised by polygons
    \begin{itemize}
        \item \(\conv((0,\frac12), (1,\frac12), (\frac12,\frac{i_2+2}{2}))\).
    \end{itemize}
    
    If \(b_1=1\) the third condition is satisfied by the tuples \((1,b_2,i_2)\) where \(3 \leq b_2 \leq 6\) and \(b_2 \leq 2i_2+6\).
    These are realised by the polygons
    \begin{itemize}
    \item \(\conv((0,0), (1,\frac12), (\frac12,\frac{i_2+1}{2}))\) if \(b_2=3\),
    \item \(\conv((0,0), (0,\frac12), (1,\frac12), (\frac12,\frac{i_2+1}{2}))\) if \(b_2=4\),
    \item \(\conv((0,0), (0,\frac12), (\frac12,0), (1,\frac12), (\frac12,\frac{i_2+1}{2}))\) if \(b_2=5\), and 
    \item \(\conv((0,-\frac12), (0,\frac12), (1,\frac12), (\frac12,\frac{i_2+1}{2}))\) if \(b_2=6\).
    \end{itemize}

    If \(b_1\geq 2\) the third condition is satisfied by the tuples \((b_1,b_2,i_2)\) where \(2b_1 \leq b_2 \leq 2b_1+4\) and \(b_2 \leq 2i_2+6\).
    These are realised by the polygons
    \begin{itemize}
    \item \(\conv((0,0),(0,b_1-2),(1,0),(\frac12,\frac{i_2+1}{2}))\) if \(b_2=2b_1\),
    \item \(\conv((0,-\frac12),(0,b_1-2),(1,0),(\frac12, -\frac12), (\frac12,\frac{i_2+1}{2}))\) if \(b_2=2b_1+1\),
    \item \(\conv((0,-\frac12),(0,b_1-2),(1,0),(1, -\frac12), (\frac12,\frac{i_2+1}{2}))\) if \(b_2=2b_1+2\),
    \item \(\conv((0,-\frac12),(0,b_1-2),(1,\frac12),(1, -\frac12), (\frac12,\frac{i_2+1}{2}))\) if \(b_2=2b_1+3\), and 
    \item \(\conv((0,-\frac12),(0,b_1-\frac32),(1,\frac12),(1, -\frac12), (\frac12,\frac{i_2+1}{2}))\) if \(b_2=2b_1+4\).
    \end{itemize}
\end{proof}

\subsubsection{Finitely growable polygons}

We grow the finitely growable polygons with zero interior points using an adaptation of Algorithm~\ref{alg:growing}.
First we identify the minimal polygons they contain.
\begin{proposition}
    Let \(P\) be a denominator 2, finitely growable polygon with zero interior points.
    Then \(P\) contains a unique minimal polygon \(Q\) with the same size and denominator as \(P\).
    The polygon \(Q\) must be one of the following
    \[
        \begin{matrix}\conv((0,0),(\tfrac12,0),(0,\tfrac12)) &
        \conv((\tfrac12,\tfrac12),(\tfrac12,1),(1,\tfrac12))\\
        \conv((0,0),(1,0),(0,\tfrac12)) &
        \conv((0,0),(1,0),(0,1))\\
        \conv((0,0),(1,0),(0,1),(1,1))&
        \conv((0,0),(1,0),(0,2))\\
        \conv((0,0),(1,0),(0,2),(1,1))&
        \conv((0,0),(1,0),(0,3))\\
        \conv((0,0),(1,0),(0,3),(1,1))&
        \conv((0,0),(2,0),(0,2))
        \end{matrix}
    \]
\end{proposition}
\begin{proof}
    Let \(P'\) denote the convex hull of the lattice points in \(P\).
    As in Section~\ref{sec:minimal_polygons} we can determine \(Q\) from \(P'\).
    If \(P'\) is empty then \(Q\) is equivalent to \(\conv((\frac12,\frac12),(\frac12,1),(1,\frac12))\).
    If \(P'\) is a point then \(Q\) is equivalent to \(\conv((0,0),(\frac12,0),(0,\frac12)\).
    If \(P'\) is a line segment then \(Q\) is equivalent to \(\conv((0,0),(k,0),(0,\frac12))\) for some positive integer \(k\).
    Finally, if \(P'\) is a lattice polygon then \(Q=P'\).
    
    If \(P'\) is a line segment of lattice length at least 2 then we may assume it contains the points \((0,0)\) and \((2,0)\).
    Since \(P\) is finitely growable we may assume that it contains a point with \(y\)-coordinate at least \(\frac32\).
    However, all such points are contained in the penumbra \(\pen(\conv((0,0),(2,0)),(n,1))\) for some integer \(n\) so none can be included in \(P\) without contradicting the condition that \(P'\) is a line segment.
    This is enough to complete the cases where \(P'\) has dimension less than 2.

    Similarly, if \(P'\) contains a line segment of lattice length at least \(4\) then, after some affine map, we may assume it contains the points \((0,0)\) and \((4,0)\).
    Since \(P\) is finitely growable we may assume it contains a point with \(y\)-coordinate at least \(\frac32\).
    However, all such points are contained in the interior of the penumbra \(\pen(\conv((0,0),(4,0)),(n,1))\) for some integer \(n\) so none can be included in \(P\) without contradicting the condition that \(P\) has no interior points.
    Therefore, \(P'\) contains at most 4 colinear lattice points.

    If \(P'\) is a lattice polygon then it is either \(\conv((0,0),(2,0),(0,2))\) or it has width 1 as these are the only lattice polygons with no interior points.
    If \(P'=\conv((0,0),(2,0),(0,2))\) then it cannot be grown by any point of \(\frac12\ZZ^2\) without including a lattice point in the interior, since all facets of \(P'\) have a lattice point in their relative interiors.
    Therefore, in this case \(P=P'\) and we do not attempt to grow \(P'\).

    If \(P'\) has width 1 then we may assume it is equivalent to a subset of the strip \([0,1] \times \RR\) and that \(P\) contains some rational vertex with \(x\)-coordinate greater than \(1\).
    This means that the intersection of \(P'\) with the line \(x=1\) must contain at most two lattice points.
    This, combined with the fact that \(P'\) contains at most 4 colinear lattice points is sufficient to classify the minimal polygons \(Q\).
\end{proof}

For each minimal polygon we apply Algorithm~\ref{alg:growing} with the new condition that we discard polygons with interior lattice points.
By the proof of Proposition~\ref{prop:colinear_points_bound}, if a finitely growable denominator 2 polygon contains at least \(4(k+1)\) half-integral points in the line \(x=\frac12\) we may assume it contains at least \(k+1\) lattice points in the line \(x=1\) and some point with \(x\)-coordinate greater than 1.
Therefore, it contains at least \(k-1\) interior points so our polygons contain at most \(8\) colinear points of \(\frac12\ZZ^2\).
We use this bound to improve compute time.
The result is a classification of 79 rational polygons \(P\) for which the tuples \((b(P),b(2P),i(2P))\) are listed in Table~\ref{tab:zero_int_pts_ehr} and plotted in Figure~\ref{fig:0_int}.
\begin{table}[ht]
    \centering
\begin{tabular}{c c c c c}
    ( 0, 3, 3 )&
    ( 1, 4, 3 )&
    ( 1, 4, 4 )&
    ( 1, 5, 3 )&
    ( 1, 6, 1 )\\
    ( 2, 4, 3 )&
    ( 2, 5, 3 )&
    ( 2, 5, 4 )&
    ( 2, 6, 3 )&
    ( 2, 6, 4 )\\
    ( 2, 7, 1 )&
    ( 2, 7, 3 )&
    ( 2, 8, 1 )&
    ( 3, 6, 3 )&
    ( 3, 6, 4 )\\
    ( 3, 7, 3 )&
    ( 3, 7, 4 )&
    ( 3, 8, 3 )&
    ( 3, 8, 4 )&
    ( 3, 9, 1 )\\
    ( 3, 9, 3 )&
    ( 4, 8, 3 )&
    ( 4, 8, 4 )&
    ( 4, 8, 5 )&
    ( 4, 9, 3 )\\
    ( 4, 9, 4 )&
    ( 4, 10, 3 )&
    ( 4, 10, 4 )&
    ( 5, 10, 3 )&
    ( 5, 10, 4 )\\
    ( 5, 11, 3 )&
    ( 5, 11, 4 )&
    ( 6, 12, 3 )&
    ( 6, 12, 4 )
\end{tabular}
    \caption{Tuples \((b(P),b(2P),i(2P))\) for all finitely growable denominator 2 polygons with zero interior points.}
    \label{tab:zero_int_pts_ehr}
\end{table}

\subsection{Polygons With Interior Points}
\label{sec:ehr_int_pnts}

We now consider the polygons with at least one interior point.
The number of boundary and interior points of \(P\) and \(2P\) for all denominator 2 polygons \(P\) containing an interior point and up to \(4\) lattice points are plotted in Figure~\ref{fig:data_plots}.
The dashed lines in each plot denote the bounds \(b(2P) \geq \max\{3,2b(P)\}\), \(i(2P) \geq b(P) + 2i(P)-1\) and \(b(2P)+i(2P) \leq 2b(P)+6i(P)+7\) all of which are satisfied by every polygon in the dataset.

Before we prove these bounds we provide some justification for our interest in them.
From Figure~\ref{fig:data_plots} we can see that the bounds are sharp.
However, since we seek a bound which is true apart from finitely many exceptions, we would like something stronger than this; each of the hyperplanes \(b(2P)=3\), \(b(2P)=2b(P)\), \(i(2P) = b(P)+2i(P)-1\) and \(b(2P)+i(2P)=2b(P)+6i(P)+7\), contains infinitely many points of the form \((b(P),i(P),b(2P),i(2P))\) for some denominator 2 polygon \(P\).
Moreover, they contain infinitely many points outside of any given plane.
Families of polygons which realise this fact can be found in Figure~\ref{fig:inf_families}.
Each family defines a sequence of colinear points, and a pair of such lines contained in each boundary are skew to one-another.
This can be seen since the lines have different direction vectors and do not intersect.
The result of this  is that there is no way to strengthen the above bounds without excluding infinitely many points.
This does not prevent the existence of additional bounds with different normal vectors, but it does suggest that an optimal polyhedron, containing all but finitely many of our points, will have unbounded facets defined by these inequalities.

To prove that these bounds hold in all but finitely many cases we first need to recall the definition of second width.
If the width of a polygon \(P\) is \(w_1=\width_{u_1}(P)\), then the \emph{second width} of \(P\) is the minimum width of \(P\) with respect to a dual vector \(u_2\), which is linearly independent to \(u_1\).
The second width is greater than or equal to the first width.
It was proven in \cite[Proposition~2.2 and 2.4]{triangles} that a polygon with first and second widths \(w_1\) and \(w_2\) is affine equivalent to a subset of the rectangle \(Q_{w_1,w_2} = [0,w_1] \times [0,w_2]\) and no sub-rectangle.
Such a polygon has a vertex on each edge of this rectangle, and its width with respect to any non-zero dual vector other than \(u_1\) is at least \(w_2\).

\begin{figure}[p!]
\centering
\subfloat[\(P\) has 1 interior point and \newline \hspace*{1.5em} 0 boundary points]{
    \begin{tikzpicture}
    \begin{axis}[width=0.45\textwidth, axis lines=left, xmin=0, ymin=0, xlabel = {\(|\partial 2P \cap \ZZ^2|\)}, ylabel = {\(|2P^\circ \cap \ZZ^2|\)}]
    \addplot[only marks, mark=*, mark size=1.5pt]
    table{data/0_boundary_1_interior.dat};
    \addplot[dashed,color=black] coordinates {(3, -1) (3, 10)};
    \addplot[domain=0:11, dashed, color=black] {1};
    \addplot[domain=3:11, dashed, color=black] {13-x};
    \end{axis}
    \end{tikzpicture}
}
\subfloat[\(P\) has 1 interior point and \newline \hspace*{1.5em} 1 boundary point]{
    \begin{tikzpicture}
    \begin{axis}[width=0.45\textwidth, axis lines=left, xmin=0, ymin=0, xlabel = {\(|\partial 2P \cap \ZZ^2|\)}, ylabel = {\(|2P^\circ \cap \ZZ^2|\)}]
    \addplot[only marks, mark=*, mark size=1.5pt]
    table{data/1_boundary_1_interior.dat};
    \addplot[dashed,color=black] coordinates {(3, -1) (3, 11)};
    \addplot[domain=0:13, dashed, color=black] {2};
    \addplot[domain=3:13, dashed, color=black] {15-x};
    \end{axis}
    \end{tikzpicture}
}\quad
\subfloat[\(P\) has 2 interior points and \newline \hspace*{1.5em} 0 boundary points]{
    \begin{tikzpicture}
    \begin{axis}[width=0.45\textwidth, axis lines=left, xmin=0, ymin=0, xlabel = {\(|\partial 2P \cap \ZZ^2|\)}, ylabel = {\(|2P^\circ \cap \ZZ^2|\)}]
    \addplot[only marks, mark=*, mark size=1.5pt]
    table{data/0_boundary_2_interior.dat};
    \addplot[dashed,color=black] coordinates {(3, -1) (3, 16)};
    \addplot[domain=0:15, dashed, color=black] {3};
    \addplot[domain=3:15, dashed, color=black] {19-x};
    \end{axis}
    \end{tikzpicture}
}
\subfloat[\(P\) has 1 interior point and  \newline \hspace*{1.5em}2 boundary points]{
    \begin{tikzpicture}
    \begin{axis}[width=0.45\textwidth, axis lines=left, xmin=0, ymin=0, xlabel = {\(|\partial 2P \cap \ZZ^2|\)}, ylabel = {\(|2P^\circ \cap \ZZ^2|\)}]
    \addplot[only marks, mark=*, mark size=1.5pt]
    table{data/2_boundary_1_interior.dat};
    \addplot[dashed,color=black] coordinates {(4, -1) (4, 14)};
    \addplot[domain=0:14, dashed, color=black] {3};
    \addplot[domain=3:14, dashed, color=black] {17-x};
    \end{axis}
    \end{tikzpicture}
}\quad
\subfloat[\(P\) has 2 interior points and \newline \hspace*{1.5em} 1 boundary point]{
    \begin{tikzpicture}
    \begin{axis}[width=0.45\textwidth, axis lines=left, xmin=0, ymin=0, xlabel = {\(|\partial 2P \cap \ZZ^2|\)}, ylabel = {\(|2P^\circ \cap \ZZ^2|\)}]
    \addplot[only marks, mark=*, mark size=1.5pt]
    table{data/1_boundary_2_interior.dat};
    \addplot[dashed,color=black] coordinates {(3, -1) (3, 18)};
    \addplot[domain=0:17, dashed, color=black] {4};
    \addplot[domain=3:17, dashed, color=black] {21-x};
    \end{axis}
    \end{tikzpicture}
}
\subfloat[\(P\) has 3 interior points and \newline \hspace*{1.5em} 0 boundary points]{
    \begin{tikzpicture}
    \begin{axis}[width=0.45\textwidth, axis lines=left, xmin=0, ymin=0, xlabel = {\(|\partial 2P \cap \ZZ^2|\)}, ylabel = {\(|2P^\circ \cap \ZZ^2|\)}]
    \addplot[only marks, mark=*, mark size=1.5pt]
    table{data/0_boundary_3_interior.dat};
    \addplot[dashed,color=black] coordinates {(3, -1) (3, 22)};
    \addplot[domain=0:20, dashed, color=black] {5};
    \addplot[domain=3:20, dashed, color=black] {25-x};
    \end{axis}
    \end{tikzpicture}
}
\caption{Number of boundary and interior points for \(P\) and \(2P\) for denominator 2 polygons \(P\) of size up to 4. Continued on next page.
}
\label{fig:data_plots}
\end{figure}
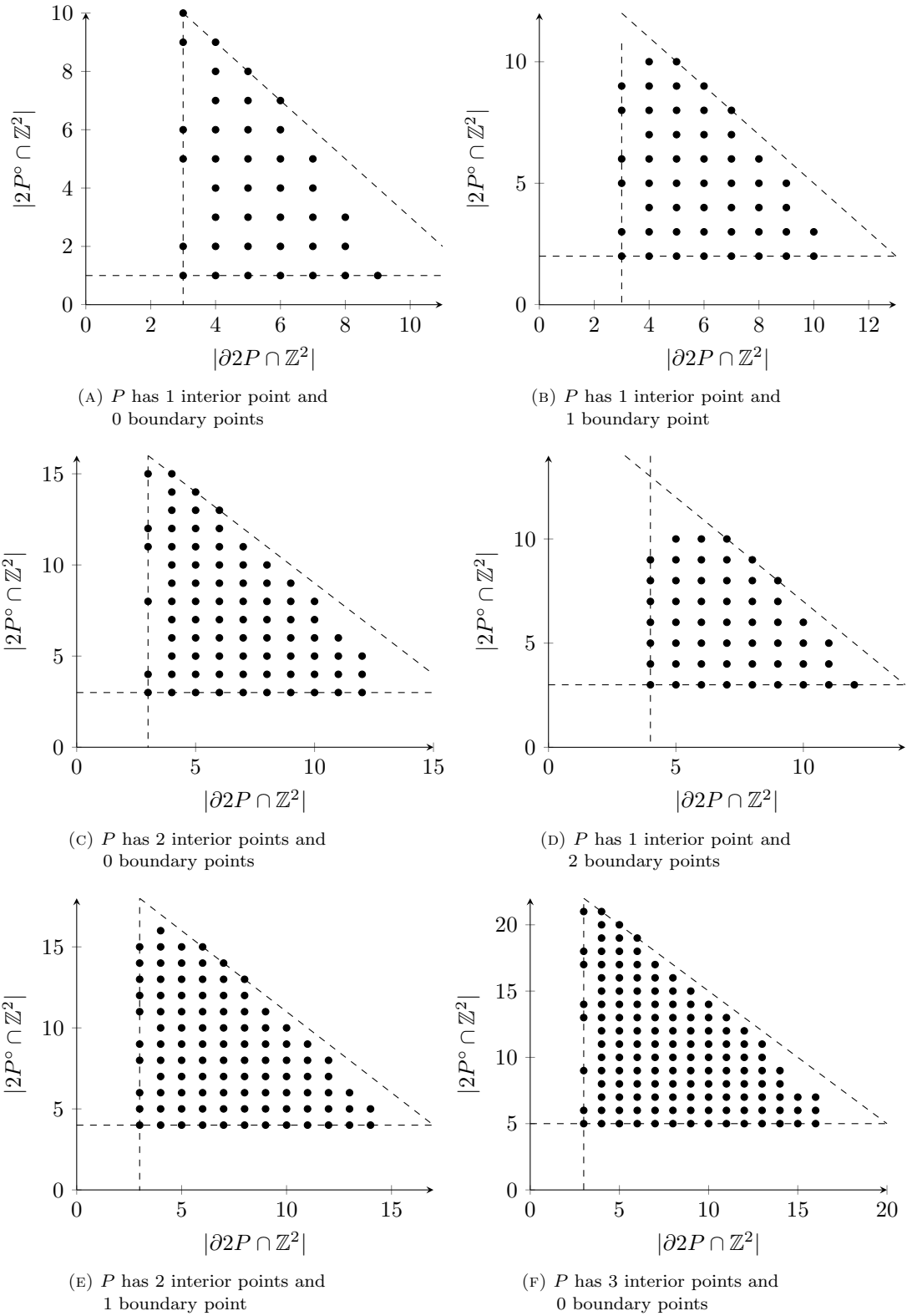

\begin{figure}[ht]
\centering
\ContinuedFloat
\subfloat[\(P\) has 1 interior point and \newline \hspace*{1.5em} 3 boundary points]{
    \begin{tikzpicture}
    \begin{axis}[width=0.45\textwidth, axis lines=left, xmin=0, ymin=0, xlabel = {\(|\partial 2P \cap \ZZ^2|\)}, ylabel = {\(|2P^\circ \cap \ZZ^2|\)}]
    \addplot[only marks, mark=*, mark size=1.5pt]
    table{data/3_boundary_1_interior.dat};
    \addplot[dashed,color=black] coordinates {(6, -1) (6, 13)};
    \addplot[domain=0:16, dashed, color=black] {4};
    \addplot[domain=6:16, dashed, color=black] {19-x};
    \end{axis}
    \end{tikzpicture}
}
\subfloat[\(P\) has 2 interior points and \newline \hspace*{1.5em} 2 boundary points]{
    \begin{tikzpicture}
    \begin{axis}[width=0.45\textwidth, axis lines=left, xmin=0, ymin=0, xlabel = {\(|\partial 2P \cap \ZZ^2|\)}, ylabel = {\(|2P^\circ \cap \ZZ^2|\)}]
    \addplot[only marks, mark=*, mark size=1.5pt]
    table{data/2_boundary_2_interior.dat};
    \addplot[dashed,color=black] coordinates {(4, -1) (4, 19)};
    \addplot[domain=0:18, dashed, color=black] {5};
    \addplot[domain=3:18, dashed, color=black] {23-x};
    \end{axis}
    \end{tikzpicture}
}\quad
\subfloat[\(P\) has 3 interior points and \newline \hspace*{1.5em} 1 boundary point]{
    \begin{tikzpicture}
    \begin{axis}[width=0.45\textwidth, axis lines=left, xmin=0, ymin=0, xlabel = {\(|\partial 2P \cap \ZZ^2|\)}, ylabel = {\(|2P^\circ \cap \ZZ^2|\)}]
    \addplot[only marks, mark=*, mark size=1.2pt]
    table{data/1_boundary_3_interior.dat};
    \addplot[dashed,color=black] coordinates {(3, -1) (3, 25)};
    \addplot[domain=0:20, dashed, color=black] {6};
    \addplot[domain=3:20, dashed, color=black] {27-x};
    \end{axis}
    \end{tikzpicture}
}
\subfloat[\(P\) has 4 interior points and \newline \hspace*{1.5em} 0 boundary points]{
    \begin{tikzpicture}
    \begin{axis}[width=0.45\textwidth, axis lines=left, xmin=0, ymin=0, xlabel = {\(|\partial 2P \cap \ZZ^2|\)}, ylabel = {\(|2P^\circ \cap \ZZ^2|\)}]
    \addplot[only marks, mark=*, mark size=1.2pt]
    table{data/0_boundary_4_interior.dat};
    \addplot[dashed,color=black] coordinates {(3, -1) (3, 29)};
    \addplot[domain=0:24, dashed, color=black] {7};
    \addplot[domain=2:24, dashed, color=black] {31-x};
    \end{axis}
    \end{tikzpicture}
}
\caption{Ehrhart data for denominator 2 polygons of size up to 4 continued.}
\end{figure}
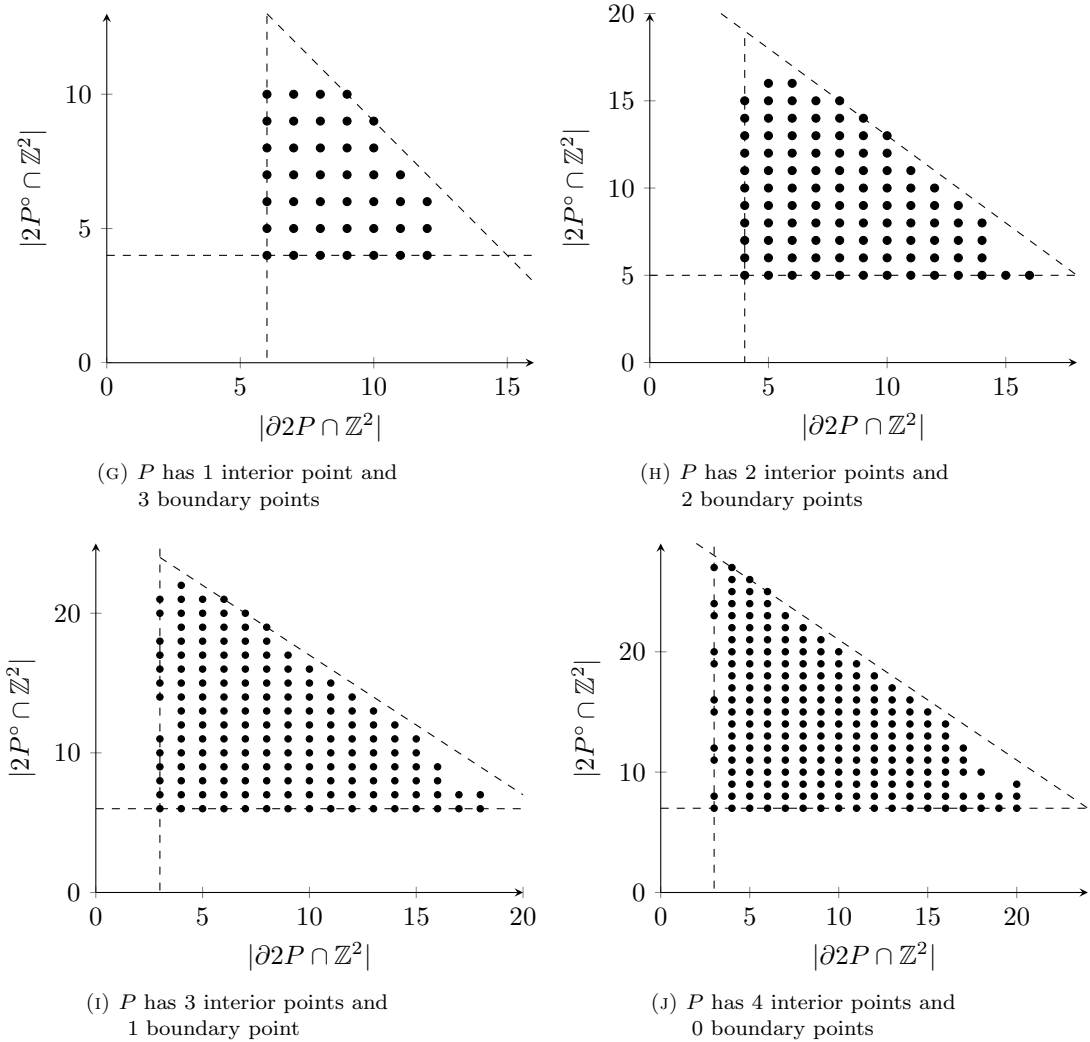

\begin{figure}[ht]
    \centering
\begin{tikzpicture}[x=0.4cm,y=0.4cm]

\foreach \i in {1,2,3}{
    \draw[fill=gray!30] (1+\i*\i + 3*\i - 4,-1+38) -- (2+\i*\i + 3*\i - 4,1+38) -- (2*\i+1+\i*\i + 3*\i - 4,0+38) -- cycle;

    \draw[fill=gray!30] (\i*\i + 3*\i - 4+0,1+32) -- (\i*\i + 3*\i - 4+0,-1+32) -- (\i*\i + 3*\i - 4+2*\i+2,0+32) -- cycle;

    \draw[fill=gray!30] (0+\i*\i + 3*\i - 4,-1+26) -- (0+\i*\i + 3*\i - 4,1+26) -- (2*\i+2+\i*\i + 3*\i - 4,1+26) -- (2*\i+2+\i*\i + 3*\i - 4,-1+26) -- cycle;

    \draw[fill=gray!30] (1+\i*\i + 3*\i - 4,-2+20) -- (2+\i*\i + 3*\i - 4,1+20) -- (2*\i+2+\i*\i + 3*\i - 4,0+20) -- cycle;

    \draw[fill=gray!30] (2+\i*\i + 3*\i - 4,2+14) -- (0+\i*\i + 3*\i - 4,0+14) -- (2+\i*\i + 3*\i - 4,-2+14) -- (2*\i+2+\i*\i + 3*\i - 4,0+14) -- cycle;

    \draw[fill=gray!30] (1+\i*\i + 3*\i - 4,2+8) -- (0+\i*\i + 3*\i - 4,1+8) -- (1+\i*\i + 3*\i - 4,-1+8) -- (3+\i*\i + 3*\i - 4,-2+8) -- (2*\i+1+\i*\i + 3*\i - 4,-2+8) -- (2*\i+2+\i*\i + 3*\i - 4,-1+8) -- (2*\i+1+\i*\i + 3*\i - 4,1+8) -- (2*\i-1+\i*\i + 3*\i - 4,2+8) -- cycle;

    \draw[fill=gray!30] (3+\i*\i + 3*\i - 4,0) -- (1+\i*\i + 3*\i - 4,0) -- (0+\i*\i + 3*\i - 4,1) -- (0+\i*\i + 3*\i - 4,3) -- (1+\i*\i + 3*\i - 4,4) -- (3+\i*\i + 3*\i - 4,4) -- (2*\i+2+\i*\i + 3*\i - 4,3) -- (2*\i+2+\i*\i + 3*\i - 4,1) -- cycle;
}

\foreach \y in {5,11,...,35}{
    \draw[dashed] (-1,\y) -- (23,\y);
}

\foreach \x in {-1,...,23}{
    \foreach \y in {0,1,...,4, 6,7,...,10, 12,13,...,16, 18,19,...,22, 24,25,...,28, 30,31,...,34, 36,37,...,40}{
        \node[draw,circle,inner sep=0.1pt,fill] at (\x,\y) {};
    }
}
\foreach \x in {0,2,...,22}{
    \foreach \y in {0,2,...,40}{
        \node[cross=2pt] at (\x,\y) {};
    }
}

\foreach \i in {1,2,3}{
\node[draw,circle,inner sep=1.5pt,fill] at (1+\i*\i + 3*\i - 4,-1+38) {};
\node[draw,circle,inner sep=1.5pt,fill] at (2+\i*\i + 3*\i - 4,1+38) {};
\node[draw,circle,inner sep=1.5pt,fill=white] at (2*\i+1+\i*\i + 3*\i - 4,0+38) {};

\node[draw,circle,inner sep=1.5pt,fill] at (\i*\i + 3*\i - 4+0,1+32) {};
\node[draw,circle,inner sep=1.5pt,fill] at (\i*\i + 3*\i - 4+0,-1+32) {};
\node[draw,circle,inner sep=1.5pt,fill=white] at (\i*\i + 3*\i - 4+2*\i+2,0+32) {}; 

\node[draw,circle,inner sep=1.5pt,fill] at (0+\i*\i + 3*\i - 4,-1+26) {};
\node[draw,circle,inner sep=1.5pt,fill] at (0+\i*\i + 3*\i - 4,1+26) {};
\node[draw,circle,inner sep=1.5pt,fill=white] at (2*\i+2+\i*\i + 3*\i - 4,1+26) {}; 
\node[draw,circle,inner sep=1.5pt,fill=white] at (2*\i+2+\i*\i + 3*\i - 4,-1+26) {};  

\node[draw,circle,inner sep=1.5pt,fill] at (1+\i*\i + 3*\i - 4,-2+20) {};
\node[draw,circle,inner sep=1.5pt,fill] at (2+\i*\i + 3*\i - 4,1+20) {};
\node[draw,circle,inner sep=1.5pt,fill=white] at (2*\i+2+\i*\i + 3*\i - 4,0+20) {}; 

\node[draw,circle,inner sep=1.5pt,fill] at (2+\i*\i + 3*\i - 4,2+14) {};
\node[draw,circle,inner sep=1.5pt,fill] at (0+\i*\i + 3*\i - 4,0+14) {};
\node[draw,circle,inner sep=1.5pt,fill] at (2+\i*\i + 3*\i - 4,-2+14) {};
\node[draw,circle,inner sep=1.5pt,fill=white] at (2*\i+2+\i*\i + 3*\i - 4,0+14) {}; 

\node[draw,circle,inner sep=1.5pt,fill] at (1+\i*\i + 3*\i - 4,2+8) {};
\node[draw,circle,inner sep=1.5pt,fill] at (0+\i*\i + 3*\i - 4,1+8) {};
\node[draw,circle,inner sep=1.5pt,fill] at (1+\i*\i + 3*\i - 4,-1+8) {};
\node[draw,circle,inner sep=1.5pt,fill] at (3+\i*\i + 3*\i - 4,-2+8) {};
\node[draw,circle,inner sep=1.5pt,fill=white] at (2*\i+1+\i*\i + 3*\i - 4,-2+8) {}; 
\node[draw,circle,inner sep=1.5pt,fill=white] at (2*\i+2+\i*\i + 3*\i - 4,-1+8) {}; 
\node[draw,circle,inner sep=1.5pt,fill=white] at (2*\i+1+\i*\i + 3*\i - 4,1+8) {}; 
\node[draw,circle,inner sep=1.5pt,fill=white] at (2*\i-1+\i*\i + 3*\i - 4,2+8) {}; 

\node[draw,circle,inner sep=1.5pt,fill] at (3+\i*\i + 3*\i - 4,0) {};
\node[draw,circle,inner sep=1.5pt,fill] at (1+\i*\i + 3*\i - 4,0) {};
\node[draw,circle,inner sep=1.5pt,fill] at (0+\i*\i + 3*\i - 4,1) {};
\node[draw,circle,inner sep=1.5pt,fill] at (0+\i*\i + 3*\i - 4,3) {};
\node[draw,circle,inner sep=1.5pt,fill] at (1+\i*\i + 3*\i - 4,4) {};
\node[draw,circle,inner sep=1.5pt,fill] at (3+\i*\i + 3*\i - 4,4) {};
\node[draw,circle,inner sep=1.5pt,fill=white] at (2*\i+2+\i*\i + 3*\i - 4,3) {}; 
\node[draw,circle,inner sep=1.5pt,fill=white] at (2*\i+2+\i*\i + 3*\i - 4,1) {}; 
}

\node[anchor=west] at (24,39) {\((0,i,3,2i-1)\)};
\node[anchor=west] at (24,38) {\(b(2P) = 3\)};
\node[anchor=west] at (24,37) {\(i(2P) = b(P) +2i(P)-1\)};

\node[anchor=west] at (24,33) {\((2, i, 4, 2i+1)\)};
\node[anchor=west] at (24,32) {\(b(2P) = 2b(P)\)};
\node[anchor=west] at (24,31) {\(i(2P) = b(P) + 2i(P) -1\)};

\node[anchor=west] at (24,27) {\((2,i,4i+8, 2i+1)\)};
\node[anchor=west] at (24,26) {\(i(2P) = b(P) +2i(P)-1\)};

\node[anchor=west] at (24,21) {\((1,i,3,3i)\)};
\node[anchor=west] at (24,20) {\(b(2P) = 3\)};

\node[anchor=west] at (24,15) {\((4,i,8,4i+1)\)};
\node[anchor=west] at (24,14) {\(b(2P) = 2b(P)\)};

\node[anchor=west] at (24,9) {\((2i-2,i,4i+2,6i+1)\)};
\node[anchor=west] at (24,8) {\(b(2P) + i(2P) = 2b(P) +6i(P) +7\)};

\node[anchor=west] at (24,3) {\((4,i,12,6i+3)\)};
\node[anchor=west] at (24,2) {\(b(2P) + i(2P) = 2b(P) +6i(P) +7\)};

\node[draw,semicircle,inner sep=0.7pt,fill] at (1,10.07) {};
\node[draw,semicircle,inner sep=0.7pt,fill] at (3,6.07) {};

\end{tikzpicture}
    \caption{Representatives of 7 infinite families of polygons, for which the tuples \((b(P),i(P),b(2P),i(2P))\) are on the boundaries of the region described in Theorem~\ref{thm:hip_finlap_main}. To obtain each whole family in general, fix the vertices denoted by filled points and successively add \((1,0)\) to the vertices denoted by hollow points. For each family, we include the general form of \((b(P),i(P),b(2P),i(2P))\) in terms of an integer \(i \geq 1\) as well as the equations of the boundaries these points are contained in.}
    \label{fig:inf_families}
\end{figure}
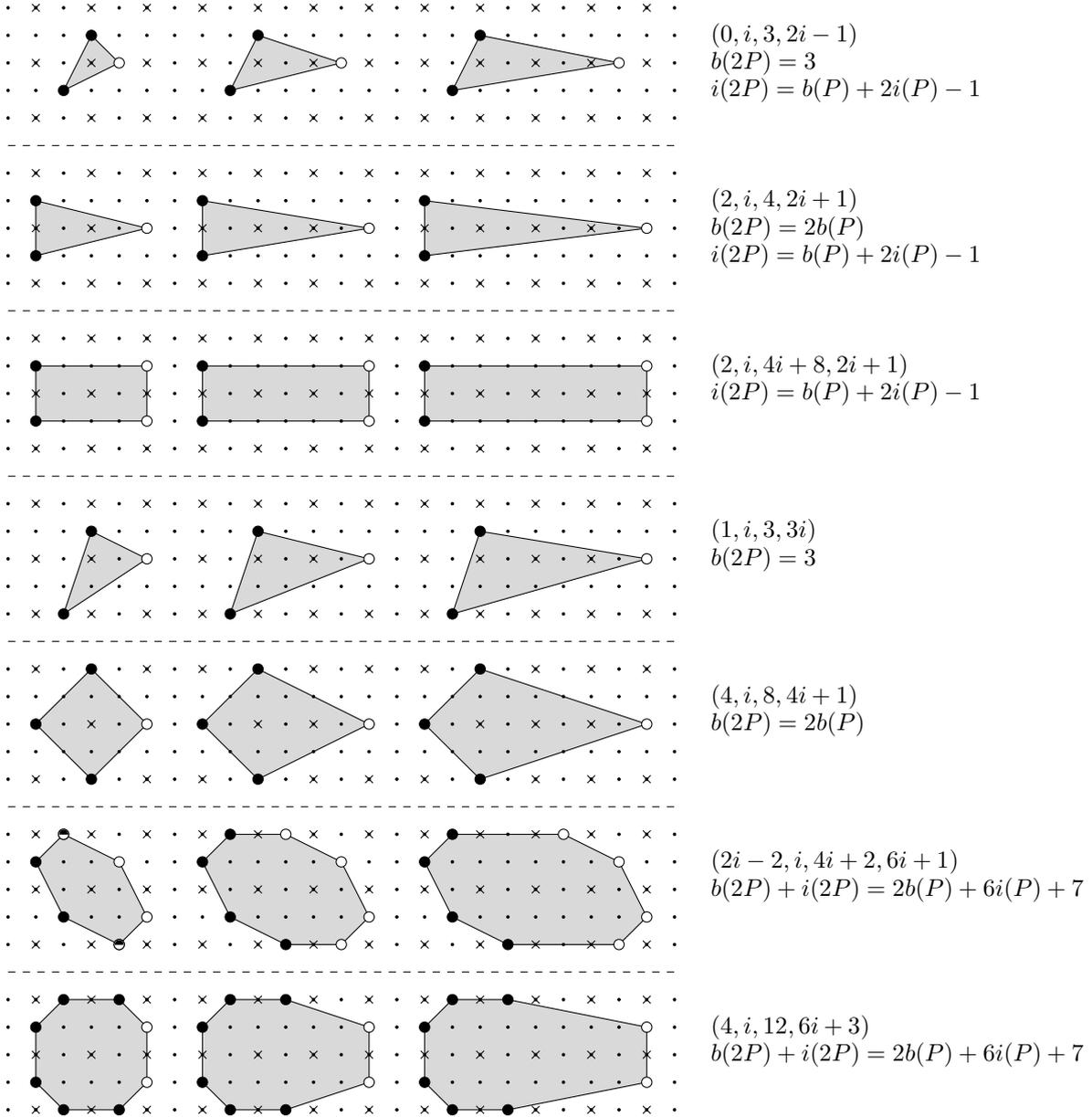

\begin{lemma}\label{lem:hourglass}
    Let \(P\) be a lattice polygon with first and second widths \(w_1\) and \(w_2\), which we assume is a subset of the rectangle \(Q_{w_1,w_2} \coloneqq [0,w_1]\times [0,w_2]\).
    Let \(h \in [2,w_1-2]\) be an integer and let \(h'\) be the minimum of \(h\) and \(w_1-h\), then \(P\) contains at least \(h'-1\) interior points in the line \(x=h\).
\end{lemma}
\begin{proof}
    Suppose for contradiction that \(P\) contains less than \(h'-1\) interior points in the line \(x=h\).
    Then the line segment \(P \cap \{x=h\}\) must be a subset of \(I_h \coloneqq \{h\}\times[y_1,y_2]\) where \(y_1\) is an integer in \([0,w_2]\) and \(y_2=y_1+h'-1\).
    
    Let \(v\) be a point of \(P\) and \(u\) a point of the line \(x=h\) not contained in this interval.
    Then no point of \(\pen(v,u)\) can be contained in \(P\).
    In particular, if \(v\) has \(x\)-coordinate less than \(h\) then all points of \(P\) with \(x\)-coordinate greater than \(h\) must be contained in the affine cone 
    \[
    C_v \coloneqq v+\cone(I_h-v)
    \]
    and the symmetric result holds when \(v\) has \(x\)-coordinate greater than \(h\).

    By the widths of \(P\) we know that \(P\) has a vertex contained in each edge of the rectangle \(Q_{w_1,w_2}\).
    Either \(y_1>0\) or \(y_2<w_2\), so the points in the upper and lower edge of \(Q_{w_1,w_2}\) cannot both be in the line \(x=h\).
    By reflections and after redefining \(h\) and \(I_h\) if necessary, we may assume that \(P\) contains a vertex \(v_0\) in the upper edge of \(Q_{w_1,w_2}\) with \(x\)-coordinate greater than \(h\).
    Let \(v\) be a point of \(P\) with \(x\)-coordinate less than \(h\), then \(v_0\) and the vertices of \(P\) on the right edge of \(Q_{w_1,w_2}\) are contained in \(C_v\).
    As a result, the point \((w_1,w_2)\) is also contained in \(C_v\), and so \(v\) is in the cone \(C_{(w_1,w_2)}\).

Suppose towards a contradiction, that \(P\) contains a point on the lower edge of \(Q_{w_1,w_2}\), with \(x\)-coordinate greater than \(h\).
Then \((w_1,0)\) must also be contained in \(C_v\) as above, so all points of \(P\) with \(x\)-coordinate less than \(h\) are contained in the intersection of \(C_{(w_1,0)}\) and \(C_{(w_1,w_2)}\).
However, the line through \((w_1,0)\) and \((h,y_1)\) and the line through \((w_1,w_2)\) and \((h,y_2)\) meet at a point with \(x\)-coordinate
\[
\frac{w_2h+w_1(y_1-y_2)}{w_2-y_2+y_1} = \begin{cases}
    \frac{w_2h-w_1(h-1)}{w_2-h+1} & \text{if \(h \leq \frac{w_1}{2}\)}\\
    \frac{w_2h-w_1(w_1-h-1)}{w_2-w_1+h+1} & \text{if \(h > \frac{w_1}{2}\)}
\end{cases}
\]
which is always greater than \(0\).
This shows that \(P\) cannot contain a point in the left edge of \(Q_{w_1,w_2}\) which is the desired contradiction.

Suppose, towards a contradiction that \(P\) contains the point \((h,0)\).
Then \(y_1=0\) and we consider \(C_{(w_1,w_2)}\) which must contain all points of \(P\) with \(x\)-coordinate less than \(h\).
However, the line through \((w_1,w_2)\) and \((h,y_2)\) intersects the \(x\)-axis at the point
\[
\left(\frac{w_2h-w_1y_2}{w_2-y_2},0\right) = 
\begin{cases}
    \left(\frac{w_2h-w_1(h-1)}{w_2-(h-1)},0\right) & \text{if \(h \leq \frac{w_1}{2}\)}\\
    \left(\frac{w_2h-w_1(w_1-h-1)}{w_2-(w_1-h-1)},0\right) & \text{if \(h > \frac{w_1}{2}\)}
\end{cases}
\]
which is always greater than 0.
As above, this prevents \(P\) from containing a point in the left edge of \(Q_{w_1,w_2}\), which is the desired contradiction.

    Now, let \(v'\) be a point of \(P\) with \(x\)-coordinate greater than \(h\).
    We have just shown that a vertex of \(P\) in the lower edge of \(Q_{w_1,w_2}\) has \(x\)-coordinate less than \(h\).
    The cone \(C_{v'}\) must contain both this vertex and some point in the left edge of \(Q_{w_1,w_2}\) so \((0,0)\) is contained in \(C_{v'}\).
    Therefore, \(v'\) is contained in \(C_{(0,0)}\).

    We have shown that \(P\) is a subset of the union of \(C_{(0,0)}\) and \(C_{(w_1,w_2)}\) as depicted in Figure~\ref{fig:hourglass}.
    We will show that the width of this region with respect to \(u=(1,-1)\) is less than \(w_2\) which is the desired contradiction.
    
\begin{figure}[ht]
\centering
\begin{tikzpicture}[x=0.9cm,y=0.9cm]
\draw (0,0) -- (5,0) -- (5,6) -- (0,6) -- cycle;
\filldraw[fill=gray!30] (0,0) -- (7/5,0) -- (2,1) -- (5,2.5) -- (5,6) -- (4,6) -- (2,3) -- (0,1) -- cycle;
\fill[fill=gray!50] (0,0) -- (2,1) -- (5,6) -- (2,3) -- cycle;
\draw[dashed] (2,0) -- (2,6);
\draw[dashed] (2,1) -- (0,0) -- (2,3);
\draw[dashed] (2,1) -- (5,6) -- (2,3);
\draw[very thick] (2,1) -- (2,3);
\node[anchor=north west] at (2,1) {\((h,y_1)\)};
\node[anchor=south east] at (2,3) {\((h,y_2)\)};
\foreach \p in {(0,1),(7/5,0),(4,6),(5,2.5)}{
    \node[draw,circle,inner sep=1pt,fill] at \p { };
}
\node[anchor=east] at (0,1) {\(p_1\)};
\node[anchor=north] at (7/5,0) {\(p_2\)};
\node[anchor=south] at (4,6) {\(p_3\)};
\node[anchor=west] at (5,2.5) {\(p_4\)};
\node[anchor=west] at (2,2.3) {\(I_h\)};
\end{tikzpicture}
\caption{In the proof of Lemma~\ref{lem:hourglass}, the polytope \(P\) is contained in the union of the two shaded affine cones \(C_{(0,0)}\) and \(C_{(w_1,w_2)}\) whose points are at \((0,0)\) and \((w_1,w_2)\) respectively.}
\label{fig:hourglass}
\end{figure}
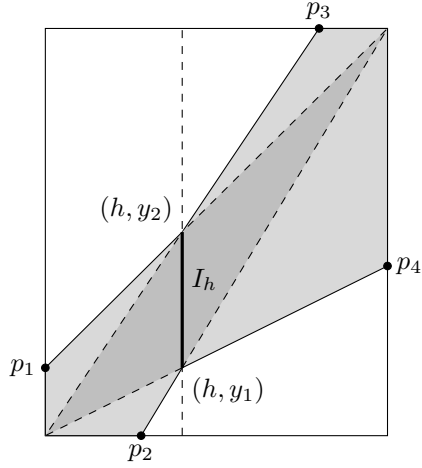

    By reflections, we can now assume that \(h'=h\).
    Notice that the line \(y=\frac{w_2}{w_1}x\) intersects the interval \(I_h\) non-trivially.
    In particular, \(w_1y_1 \leq w_2h \leq w_1y_2\).
    Consider the slope of each edge of \(C_{(0,0)}\) and \(C_{w_1,w_2}\) compared to the line \(y=\frac{w_2}{w_1}x\).
    The edges meeting at \((h,y_1)\) and \((h,y_2)\) have gradients such that these two points can never be a vertex of the convex hull \(P'\) defined
    \[
    P'\coloneqq\conv\left(\left(C_{(0,0)} \cup C_{(w_1,w_2)}\right) \cap Q_{w_1,w_2}\right).
    \]
    Additionally, consider the points \(p_1,\dots, p_4\) in Figure~\ref{fig:hourglass}.
    They are
    \[
    \begin{matrix}
    p_1 = (0,w_2-\tfrac{w_2-y_2}{w_1-h}w_1), & p_2 = (w_1-\tfrac{w_1-h}{w_2-y_1}w_2,0), \\
    p_3 = (\tfrac{h}{y_2}w_2,w_2), & p_4 = (w_1,\tfrac{y_1}{h}w_1).
    \end{matrix}
    \]
    Since the value \(u\) takes increases from left to right along horizontal lines and from top to bottom along vertical ones, \(u \cdot p_1 \leq u \cdot (0,0) \leq u \cdot p_2\) and similarly \(u \cdot p_3 \leq u\cdot (w_1,w_2) \leq u\cdot p_4\).
    Therefore, we need only show that the difference between each pair of values of \(u \cdot p_i\) is less than \(w_2\).

    By considering the value of \(u \cdot (0,0)\) and \(u\cdot (w_1,w_2)\) we see that \(u \cdot p_1\) and \(u \cdot p_3\) are at most zero and that \(u \cdot p_2\) is at least 0.
    If \(u\cdot p_i\) and \(u\cdot p_j\) are both positive or both negative, we need not check \(u \cdot(p_i-p_j)\) or \(u\cdot(p_j-p_i)\) since one of the points will always have a greater difference with some other point (a negative or positive one respectively).
    We now prove that each of the following is less than \(w_2\) by contradiction which completes the proof.
    \[
    u\cdot(p_2-p_1), \quad u\cdot(p_2-p_3), \quad u\cdot(p_2-p_4), \quad u \cdot(p_4-p_1), \quad u\cdot(p_4-p_3).
    \]

    If \(u\cdot(p_2-p_1) \geq w_2\) then
    \[
    w_1-\tfrac{w_1-h}{w_2-y_1}w_2 + w_2-\tfrac{w_2-y_1-h+1}{w_1-h}w_1 \geq w_2.
    \]
    This inequality can be rearranged into the following:
    \[
    w_1y_1(2w_2-w_1-y_1+1) \geq w_2(w_1w_2 + w_1 -2w_1h+h^2).
    \]
    Since the left-hand side is positive we can use the fact that \(w_2h \geq w_1y_1\) to say that this is all less than or equal to \(w_2h(2w_2-w_1-y_1+1)\).
    Dividing by \(w_2\) and rearranging again shows that 
    \[
    h(2w_2-y_1+1+w_1-h) \geq w_1(w_2 + 1).
    \]
    But \(w_1 \geq 2h\) so, after canceling and rearranging again, this implies that \(-1 \geq y_1+h\) which is the desired contradiction.

    If \(u\cdot(p_2-p_3) \geq w_2\) then
    \[
    w_1-\tfrac{w_1-h}{w_2-y_1}w_2 - \tfrac{h}{y_1+h-1}w_2 + w_2 \geq w_2.
    \]
    This inequality can be rearranged into the following:
    \[
    w_2h(2y_1-w_2+h-1) \geq w_1y_1(y_1+h-1).
    \]
    The right-hand side is non-negative so we may assume the left-hand side is too.
    Using the fact that \(w_2h \leq w_1(y_1+h-1)\) we can rearrange to get
    \[
    y_1+h \geq w_2+1
    \]
    which is the desired contradiction.

    If \(u\cdot(p_2-p_4) \geq w_2\) then
    \[
    w_1-\tfrac{w_1-h}{w_2-y_1}w_2 - w_1+\tfrac{y_1}{h}w_1 \geq w_2
    \]
    which we can rearrange into
    \[
    w_1y_1(w_2-y_1) \geq w_2h(w_2-y_1+w_1-h).
    \]
    Since we know that \(w_2h \geq w_1y_1\) and the right-hand side is non-negative we get
    \[
    w_2-y_1 \geq w_2-y_1+w_1-h
    \]
    which is a contradiction since \(w_1 > h\).

    If \(u\cdot(p_4-p_1) \geq w_2\) then
    \[
    w_1-\tfrac{y_1}{h}w_1-\tfrac{w_2-y_1-h+1}{w_1-h}w_1+w_2 \geq w_2
    \]
    which can be rearranged into
    \[
     h(w_1-w_2-1) \geq y_1(w_1-2h).
    \]
    The right-hand side is non-negative and the left-hand side is negative which is the desired contradiction.

    If \(u\cdot(p_4-p_3) \geq w_2\) then
    \[
    w_1-\tfrac{y_1}{h}w_1 - \tfrac{h}{y_1+h-1}w_2+w_2 \geq w_2
    \]
    which can be rearranged into
    \[
    w_1h(y_1+h-1) \geq w_2h^2+w_1y_1(y_1+h-1)
    \]
    but \(w_1(y_1+h-1) \geq w_2h\) so we get
    \[
    w_1h(y_1+h-1) \geq w_2h(h+y_1)
    \]
    which is the desired contradiction.
\end{proof}

We introduce new notation for the following proofs.
Let \(p_h(P)\), \(b_h(P)\) and \(i_h(P)\) denote the number of points, boundary points and interior points of \(P\) in the line \(x=h\).

\begin{lemma}\label{lem:trapezium}
    Let \(P\) be a lattice polygon with vertices \((0,y_1)\), \((0,y_2)\), \((2,y_3)\), \((2,y_4)\) where \(y_1 \leq y_2\) and \(y_3 \leq y_4\), then 
    \[
    p_0(P) + p_2(P) = p_1(P)+i_1(P)+2.
    \]
\end{lemma}
\begin{proof}
    The normalised volume of \(P\) is \(2(y_2-y_1+y_4-y_3)\) and the number of boundary points of \(P\) is \(y_2-y_1+y_4-y_3+2+b_1(P)\).
    Using Pick's theorem we can combine these to show that the number if interior points of \(P\) is \(\frac12(y_2-y_1+y_4-y_3) - \frac12 b_1(P)\).
    The number of points of \(P\) in \(x=0\) and \(x=2\) combined is equal to the number of boundary points of \(P\) minus \(b_1(P)\) so
    \[
    p_0(P) + p_2(P) = y_2-y_1+y_4-y_3+2 = 2i(P) +b_1(P)+2
    \]
    which equals \(p_1(P)+i_1(P)+2\) since \(i(P) = i_1(P)\).
\end{proof}

\begin{lemma}\label{lem:hourglass_finiteness}
    Fix integers \(w > 1\), \(k > 0\) and \(1 \leq h \leq w-1\).
    Up to affine equivalence, there are finitely many polygons \(P\) contained in the strip \([0,w] \times \RR\) with width \(w\) and less than \(k\) lattice points in the line \(x=h\).
\end{lemma}
\begin{proof}
    By a shear about the \(y\)-axis and a translation we may assume that the intersection of \(P\) and the line \(x=h\) is a subset of the interval \(I_h \coloneqq \{h\} \times [0,k]\) and that \(P\) has a vertex \(v_0 = (0,y_0)\) with \(y_0 \in [0,h-1]\).
    As in the proof of Lemma~\ref{lem:hourglass} all points of \(P\) with \(x\)-coordinate greater than \(h\) must be contained in the cone \(C_{v_0} \coloneqq v_0+\cone(I_h-v_0)\).
    Since \(P\) has width \(w\), it must have some vertex \(v_w=(w,y_w)\) contained in the interval where \(C_{v_0}\) intersects the line \(x=w\).
    Once again, all points of \(P\) with \(x\)-coordinate less than \(h\) must be contained in the cone \(C_{v_w}\).

    There are finitely many choices for \(v_0\) and finitely many choices for \(v_w\) in each case.
    The conditions on \(P\) described above define a finite region which must contain \(P\).
    Thus, there are only finitely many such polygons \(P\).
\end{proof}
\begin{proposition}\label{prop:ehr_diag_bound}
Let \(P\) be a denominator 2, finitely growable polygon.
Apart from finitely many exceptions, \(p(2P) \leq 2b(P)+6i(P)+7\).
\end{proposition}
\begin{proof}
    By a change of basis we may assume that \(\width(P) = \width_{(1,0)}(P)\) and that \((1,0) \cdot P \subseteq [0,w]\) for some integer \(w\) such that
    \begin{itemize}
        \item[(A)] \((1,0) \cdot P = [0,w]\),
        \item[(B)] \((1,0) \cdot P = [0,w-\frac12]\) or
        \item[(C)] \((1,0) \cdot P = [\frac12,w-\frac12]\).
    \end{itemize}
    We prove the bound by bounding the number of points of \(2P\) in each line \(x=0,1,\dots,2w\).
    
    First notice that by Lemma~\ref{lem:trapezium}, for any \(h = 1, \dots, w-1\)
    \[
    p_{2h-1}(2P) + p_{2h+1}(2P) \leq p_{2h}(2P)+i_{2h}(2P) + 2.
    \]
    This allows us to eliminate positive odd terms from \(p(2P) = \sum_{h=0}^{2w} p_h(2P)\) as follows:
    \begin{align*}
    p(2P) = & \sum_{h=0}^w p_{2h}(2P) + \sum_{h=1}^{w-1}(p_{2h-1}(2P) + p_{2h+1}(2P)) - \sum_{h=1}^{w-2}p_{2h+1}(2P)\\
     \leq & \sum_{h=0}^w p_{2h}(2P) + \sum_{h=1}^{w-1}(p_{2h}(2P)+i_{2h}(2P)+2) - \sum_{h=1}^{w-2}p_{2h+1}(2P).
    \end{align*}
    
    Now observe that for any integer \(0<h<w\) the number of points \(p_{2h}(2P)\) is at most \(2i_h(P) + b_h(P) +1\) and the number of interior points \(i_{2h}(2P)\) is at most \(2i_h(P)+1\).
    We can simplify the previous inequality and substitute these bounds into it to obtain
    \[
    p(2P) \leq p_0(2P)+p_{2w}(2P)+6i(P) + 5(w-1) + \sum_{h=1}^{w-1} 2b_h(P) - \sum_{h=1}^{w-2}p_{2h+1}(2P).
    \]

    The number of points \(p_0(2P)\) and \(p_{2w}(2P)\) are bounded by \(2b_0(P)+1\) and \(2b_w(P)+1\) respectively. 
    Thus we get the following:
    \[
    p(2P) \leq 2b(P) + 6i(P) +5(w-1)+2-\sum_{h=1}^{w-2}p_{2h+1}(2P).
    \]

    If \(w=2\) then \(5(w-1)+2-\sum_{h=1}^{w-2}p_{2h+1}(2P)\) is at most \(7\), which proves the result in this case.
    For larger \(w\) we find a lower bound for \(\sum_{h=1}^{w-2}p_{2h+1}(2P)\) using Lemma~\ref{lem:hourglass}.
    This bound depends on which case out of (A), (B) and (C) we are in.
    First notice that for any odd integer \(3 \leq h \leq w-3\)
    \begin{itemize}
        \item[(A)] \(p_h(2P) \geq \begin{cases}
            h-1 & \text{if \(h \leq w\)}\\
            (2w-h) - 1& \text{if \(h > w\)}
        \end{cases}\)
        \item[(B)] \(p_h(2P) \geq \begin{cases}
            h-1 & \text{if \(h < w-\tfrac12\)}\\
            (2w-h)-2 & \text{if \(h > w-\tfrac12\)}
        \end{cases}\)
        \item[(C)] \(p_h(2P) \geq \begin{cases}
            h-2 & \text{if \(h \leq w\)}\\
            (2w-h)-2 & \text{if \(h > w\).}
        \end{cases}\)
    \end{itemize}
    We sum each of these, separating the cases where \(w\) is odd and even, to get
    \begin{itemize}
        \item[(A)] \(\sum_{h=1}^{w-2}p_{2h+1}(2P) \geq \begin{cases}
           \tfrac12w^2 - w & \text{if \(w\) even}\\
           \tfrac12w^2
           - w + \frac12 & \text{if \(w\) odd}
        \end{cases}\)
        \item[(B)] \(\sum_{h=1}^{w-2}p_{2h+1}(2P) \geq \frac12w^2 - \frac32w + 1\)
        \item[(C)] \(\sum_{h=1}^{w-2}p_{2h+1}(2P) \geq \begin{cases}
           \frac12 w^2 - 2w + 2 & \text{if \(w\) even}\\
            \frac12 w^2 - 2w + \frac52& \text{if \(w\) odd.}
        \end{cases}\)
    \end{itemize}
    Applying these bounds to \(p(2P)\) shows that
    \begin{itemize}
        \item[(A)] \(p(2P) \leq 2b(P) + 6i(P) + \begin{cases}
            - \frac12w^2 + 6w - 3 & \text{if \(w\) even}\\
            - \frac12w^2 + 6w - \frac72 & \text{if \(w\) odd}
        \end{cases}\)
        \item[(B)] \(p(2P) \leq 2b(P) + 6i(P)
            - \frac12w^2 + \frac{13}{2}w - 4\)
        \item[(C)] \(p(2P) \leq 2b(P) + 6i(P) + \begin{cases}
            - \frac12w^2 + 7w -5 & \text{if \(w\) even}\\
            - \frac12w^2 + 7w - \frac{11}{2} & \text{if \(w\) odd} 
        \end{cases}\)
    \end{itemize}
    The polynomials in \(w\) are less than or equal to \(7\) when
    \[
    \text{(A)} \quad w \geq 10, \quad \text{(B)} \quad w \geq 11, \quad \text{(C)} \quad w \geq 12.
    \]
    Therefore, if \(p(2P) > 2b(P) + 6i(P) +7\) then the width of \(2P\) is at most \(20\) and it contains less than \(5(w-1)-5\) lattice points in the hyperplanes \(x=3,5,\dots,2w-3\) in total.
    In particular, \(2P\) contains less than \(5(w-1)-5\) points in the hyperplane \(x=3\).
    By Lemma~\ref{lem:hourglass_finiteness} this shows exceptions are finite.
\end{proof}

Finally, we prove the last two bounds in Theorem~\ref{thm:hip_finlap_main}.

\begin{proposition}\label{prop:ehr_h_v_bounds}
    Let \(P\) be a polygon with denominator \(2\) and at least one interior point.
    Then \(b(2P) \geq \max\{3,2b(P)\}\) and \(i(2P) \geq b(P) + 2i(P)-1\). 
\end{proposition}
\begin{proof}
    By definition \(2P\) is a lattice polygon so \(b(2P) \geq 3\).
    If \(b(P) \geq 2\), consider two adjacent boundary lattice points of \(P\) (i.e. one can walk from one to the other along the boundary of \(P\) without touching another lattice point).
    These are either on the same edge of \(P\), or there is at least one half-integral vertex on the boundary between them.
    In either case, there is a point of \(\frac12 \ZZ\) on the boundary of \(P\) between them.
    Therefore, \(b(2P) \geq 2b(P)\) and the first inequality holds.

    Let \(Q\) be the convex hull of the interior lattice points of \(P\).
    If \(Q\) is a point or a line then \(2Q\) contains \(2i(P)-1\) points.
    Otherwise, \(2Q\) contains \(3i(P)+i(Q)-3\) 
    points by the considering the Ehrhart polynomial of \(Q\) and Pick's theorem.
    Since \(Q\) is a polygon, \(P\) has at least three interior points so \(3i(P)+i(Q)-3 \geq 2i(P) - 1\).
    Therefore, \(2Q\) contains at least \(2i(P)-1\) lattice points and all of these are interior points of \(2P\).
    
    Now, consider subdividing the region \(P \setminus Q\) in such a way that each boundary lattice point of \(P\) has a corresponding line to a boundary lattice point of \(Q\) and these lines do not intersect outside of \(Q\).
    Each such line has a half-integral point at its midpoint which is \(\frac12\) times an interior point of \(2P\) not contained in \(2Q\).
    This gives at least an additional \(b(P)\) interior points in \(2P\).
    Thus, the second inequality holds.
\end{proof}

\bibliographystyle{alpha}
\bibliography{hip-finlap}

\end{document}